\documentclass{amsart}
\usepackage{graphics}
\usepackage{amsthm}
\usepackage{amssymb}
\usepackage{amscd}
\usepackage{amsmath}
\usepackage[all]{xy}
\usepackage{color} %
\usepackage{kbordermatrix}


\def\F{\mathbb{F}}

\def\deg{{\rm deg}}

\def\Ker{{\rm Ker}}

\newcommand{\Spec}{\operatorname{Spec}}

\newcommand{\Proj}{\operatorname{Proj}}

\newtheorem{theorem}{Theorem}[section]
\newtheorem{lemma}[theorem]{Lemma}
\newtheorem{proposition}{Proposition}[section]

\theoremstyle{definition}
\newtheorem{definition}[theorem]{Definition}
\newtheorem{example}[theorem]{Example}

\newtheorem{algorithm}[theorem]{Algorithm}

\theoremstyle{remark}
\newtheorem{remark}[theorem]{Remark}

\numberwithin{equation}{section}

\begin{document}

\title[Computing the space of differential forms of a plane curve]{Computing the space of differential forms of\\ a plane curve and its Cartier-Manin matrix}


\author{Momonari Kudo}
\address{}
\curraddr{}
\email{kudo@mist.i.u-tokyo.ac.jp	}
\thanks{Graduate School of Information Science and Technology, The University of Tokyo}

\author{Shushi Harashita}
\address{}
\curraddr{}
\email{harasita@ynu.ac.jp}
\thanks{Graduate School of Environment and Information Sciences, Yokohama National University}



\date{}

\dedicatory{}


\maketitle

\begin{abstract}
In this paper, we propose a feasible algorithm to give an explicit basis of the space of regular differential forms on the nonsingular projective model of any given plane algebraic curve.
The algorithm is demonstrated for concrete examples, with our implementation over the computer algebra system Magma.
As an application, we also describe the Cartier-Manin matrix of the nonsingular projective curve with respect to the basis computed by the algorithm.
\end{abstract}

\section{Introduction}\label{sec:Intro}

\subsection{Background}
Let $p$ be a rational prime, and $k$ a perfect field of characteristic $p$.
Let $C$ be a geometrically irreducible plane curve over $k$ define by $F(x,y) = 0$ with $F \in k[x,y]$, and let ${\tilde C}$ denote the desingularization of the Zariski closure $C'$ in ${\mathbb P^2}$ of $C$.
We denote by $\varOmega_{\tilde C}$ the sheaf of regular differential forms on $\tilde{C}$, and ${\mathcal V}$ denotes the Cartier operator on the cohomology group $H^0({\tilde C},\varOmega_{\tilde C})$.
A matrix representing $\mathcal{V}$ with respect to a suitable basis for $H^0({\tilde C},\varOmega_{\tilde C})$ is called a {\it Cartier-Manin matrix} for $\tilde{C}$.
Computing a basis of $H^0({\tilde C},\varOmega_{\tilde C})$ and the Cartier-Manin matrix is a very important task both in theory and computation, since they are used to compute various invariants such as $a$-number, $p$-rank, and so on for the classification of curves.
Indeed, there are many works on this task, e.g., \cite{Manin}, \cite{Yui}, \cite{Gonzalez}, \cite{Archinard}, \cite{Bostan}, \cite{HS}, \cite{Sutherland}, \cite{OH}.
As in \cite{Gonzalez}, some works were on the first cohomology $H^1(\tilde{C},{\mathcal O}_{\tilde{C}})$ of the structure sheaf $\mathcal{O}_{\tilde{C}}$, which is the dual notion of the space of regular differential forms and on which the natural action of the Frobenius with respect to a basis is called the {\it Hasse-Witt matrix}.
For example, for the case of genus-5 and trigonal curve $C$, we compute $H^1(\tilde{C},{\mathcal O}_{\tilde{C}})$ and a Hasse-Witt matrix on it, see our preceding paper \cite{trigonal}.
The first author also obtained a generalized algorithm to compute coherent cohomologies of projective schemes and the Frobenius actions on them, see \cite{Kudo2}.

Among previous works including those described above, one of the most important results is St\"ohr-Voloch's beautiful formula~\cite[Theorem 1.1]{SV} (see \eqref{formula:Stohr-Voloch} below for a recall).
By their formula, once a basis of $H^0 (\tilde{C}, \varOmega_{\tilde{C}})$ is given explicitly, the Cartier-Manin matrix of $\tilde{C}$ with respect to the basis can be also computed.
As for the computation of a basis of $H^0 (\tilde{C}, \varOmega_{\tilde{C}})$, when $C'$ is already nonsingular, it is known that under the assumptions {\bf (A1)} and {\bf (A2)} in Section \ref{sec:main} below, $(x^i y^j  / (\partial F / \partial y)) dx$ with $0 \leq i + j \leq \mathrm{deg}(F)-3$ form a basis of $H^0({\tilde C},\varOmega_{\tilde C})$ (this fact can be viewed as a particular case of Gorenstein's result described below).
On the other hand, when $C'$ is singular, such explicit bases are found only for {\it particular} cases, e.g., \cite{Manin}, \cite{Yui} for hyperelliptic curves, \cite{Gonzalez} for Fermat curves, \cite{Sutherland} for superelliptic curves, and \cite{Archinard}, \cite{OH} for curves associated to Appell-Lauricella hypergeometic series.
For example, Sutherland recently constructed an efficient algorithm to compute Cartier-Manin matrices for superelliptic curves in \cite{Sutherland}, where he found a basis of $H^0(\tilde{C}, \varOmega_{\tilde{C}})$ to apply St\"ohr-Voloch's formula.
See also \cite{Bostan} and \cite{HS} for efficient algorithms to compute Cartier-Manin matrices for hyperelliptic curves.

\subsection{Main results}
In this paper, we shall present an algorithm for finding a basis of $H^0({\tilde C},\varOmega_{\tilde C})$ for {\it general} (i.e., possibly singular) plane curves, based on the work by Gorenstein \cite{Gorenstein}.
In \cite[Theorem 12]{Gorenstein}, the necessary and sufficient condition for a rational differential form to be regular was given in terms of conductor, but unfortunately in the statement an assumption on infinite places was added.
As Gorenstein remarked, the assumption can be removed if $k$ is infinite, but it causes a problem when we want to have a basis over $k$ when $k$ is finite.
To resolve the problem,
in Theorem \ref{thm2:Gorenstein} below
we shall give a more general statement than \cite[Theorem 12]{Gorenstein}.
Then our problem is reduced to describing the conductors
of the coordinate rings of two affine covers of the Zariski closure $C'$
in ${\mathbb P}^2$, see Subsection \ref{subsec:Gorenstein} below for Gorenstein's result and this observation.
From our observation, we construct an algorithm to compute a basis of $H^0({\tilde C},\varOmega_{\tilde C})$, by realizing the following two steps:

\begin{enumerate}
\item[{\bf Step A.}] Compute an explicit basis of the conductor $\mathfrak{C}_{k[C]}$ as a $k[x]$-module.
\item[{\bf Step B.}] Compute the truncation of $\mathfrak{C}_{k[C]}$ by the degree described in Theorem \ref{thm2:Gorenstein}.
\end{enumerate}
For Step A, we would like to apply M{\v n}uk's algorithm~\cite{Mnuk}, whose framework is:
\begin{enumerate}
\item[(M1)] Compute a basis $\{ w_1, \ldots , w_N \}$ of the integral closure $\overline{k[C]}$ as a $k[x]$-module.
\item[(M2)] By computing the trace of $ w_i w_j $ for each $1 \leq i \leq j \leq N$, construct a basis $\{ w_1^{\ast}, \ldots , w_N^{\ast} \}$ of the {\it complementary module} $\mathcal{C}_{\overline{k[C]}/k[x]}$ (see Definition \ref{def:comp_mod} below for the definition).
Note that this basis is dual to $\{ w_1, \ldots , w_N \}$, and $\{ (\partial F /\partial y) w_i^{\ast} : 1 \leq i \leq j \leq N \}$ is a desired basis of $\mathfrak{C}_{k[C]}$.
\end{enumerate}
However, we could not directly apply M\v{n}uk's algorithm for Step A, since he assumed that the characteristic of $k$ is zero.
Moreover, he utilized an algorithm~\cite{Hoeij} for (M1), but \cite{Hoeij} also assumes $\mathrm{char}(k) = 0$.
Here, our solutions for realizing Step A are:
\begin{enumerate}
\item[(i)] We confirm in Subsection \ref{subsec:complementary} that M{\v n}uk's framework works well even in the positive characteristic case.
\item[(ii)] We present an alternative approach based on the Gr\"obner basis computation with respect to plural module orders (TOP, POT and their product orders).
In our alternative approach, the integral closure $\overline{k[C]}$ is realized as an affine ring by the normalization method~\cite{GP}, and then its basis will be constructed mainly by the syzygy computation. 
\end{enumerate}
As for Step B, we compute a Gr\"{o}bner basis and the reduced row echelon form of a Macaulay matrix, see Step B of Algorithm \ref{alg:main} in Subsection \ref{subsec:main} below for details.

We implemented our algorithm over a computer algebra system, Magma~\cite{Magma}, and demonstrate it for concrete examples, which will be helpful for the further use of our algorithm.
As an application of giving a basis of the space of regular differential forms of our curve $\tilde C$, in Section \ref{sec:Cartier-Manin} we also describe the Cartier-Manin matrix of $\tilde C$ with respect to the basis, where we use the result by St\"ohr and Voloch \cite{SV}.
Furthermore, we give a remark on finding a basis of the dual space $H^1({\tilde C}, {\mathcal O}_{\tilde C})$.

\subsection{Organization}
The rest of this paper is organized as follows:
In Section \ref{sec:preliminaries}, we review Gorenstein's work and basic facts on complementary modules.
Section \ref{sec:normalization} is devoted to computations of syzygy modules and normalizations.
In Section \ref{sec:main}, we present our main algorithm, and in Section \ref{sec:implementation}, we demonstrate the algorithm on an example with our implementation over Magma~\cite{Magma}.
In Section \ref{sec:Cartier-Manin}, we explain a way to compute the Cartier-Manin matrix of a given plane curve with respect to a basis of the differential forms, computed in the previous sections.
Section \ref{sec:conc} is devoted to concluding remarks.

\section{Preliminaries}\label{sec:preliminaries}
In this section, we collect fundamental results,
which are necessary to describe our algorithm.

\subsection{Gorenstein's work}\label{subsec:Gorenstein}
We start with recalling Gorenstein's description
of the regular differential forms on the nonsingular projective model of a plane curve.

Let $C$ be a plane curve $\Spec A$, say
\[
A = k[x,y]/\langle F \rangle
\]
for an element $F=F(x,y) \in k[x,y]$.
We assume that $F$ is irreducible over the algebraic closure of $k$.
Let $N$ be the degree of $F$.
Consider the embedding of ${\mathbb A}^2=\Spec k[x,y]$ into ${\mathbb P}^2 = \Proj k[X,Y,Z]$ by $x=X/Z$ and $y=Y/Z$. Let $C'$ be the Zariski closure of $C$ in ${\mathbb P}^2$. Let $\tilde C$ be the desingularization of $C'$.
In this paper, we propose an algorithm to get
a basis of $H^0({\tilde C},\varOmega_{\tilde C})$ and to obtain
the Cartier-Manin matrix of ${\tilde C}$ with respect to the basis.

Let $k(C)$ be the function field of $C$ ($=$ that of $\tilde C$).
We may assume that
$k(C)$ is separably generated over $k$:
moreover after a coordinate change, we assume that 
\begin{enumerate}
\item[$\small\bullet$]
$x\in k(C)$ is transcendental over $k$ and 
\item[$\small\bullet$]
$F(x,y)$ is a monic polynomial in $y$ over $k[x]$ of degree $N$,
which is separable as a polynomial over $k(x)$,
\end{enumerate}
which are equivalent to {\bf (A1)} and {\bf (A2)} in Section \ref{sec:main} below.


Recall Gorenstein's result:
\begin{theorem}[{\cite[Theorem 12]{Gorenstein}}]\label{thm:Gorenstein}
Assume that $C'$ is nonsingular at any point of $C'\smallsetminus C$. Then
\[
H^0({\tilde C},\varOmega_{\tilde C}) = \left\{\left.
\frac{\phi(x,y)}{F_y} dx
\ \right|\ \phi \in {\frak C}_A,\ \deg(\phi)\leq N-3\right\}
\]
with $F_y=\frac{\partial F}{\partial y}$, where ${\frak C}_A$ is the conductor of $A=k[x,y]/\langle F \rangle$ in its integral closure $\overline A$ in $k(C)$, i.e., ${\frak C}_A = \{z\in A \mid z{\overline A} \subset A\}$.
\end{theorem}


The aim of this paper is to give an algorithm to give a basis of
the space of $\phi(x,y)$ in the theorem and to consider the case where the assumption in the theorem is removed.
Note that if $k$ is small, then we may not remove the assumption by
taking any $k$-linear coordinate-change (cf.\ Example \ref{exam:Cartier-Manin}, (3)).
Also we give implementations of these algorithms over a computer algebra system, Magma~\cite{Magma}.

The next theorem can be applicable even if $C'$ has a singularity at $C'\smallsetminus C$.
The proof has already been done essentially in the proof of \cite[Theorem 12]{Gorenstein}.
\begin{theorem}\label{thm2:Gorenstein}
With the notation above, we have
\[
H^0({\tilde C},\varOmega_{\tilde C}) = \left\{\left.
\frac{\phi(x,y)}{F_y} dx \ \right|\ 
\phi \in {\frak C}_A,\ \phi' \in {\frak C}_{A'}
\right\},
\]
where
$A'=k[x',y']/\langle F' \rangle$, where $F'(x',y')=F(x,y)/x^N$ with $(x',y')=(1/x,y/x)$, and
$\phi'(x',y'):=\phi(x,y)/x^{N-3}$.
(Note that $\phi' \in {\frak C}_{A'}$ especially requires $\deg(\phi)\le N-3$, as ${\frak C}_{A'} \subset A'$.)

\end{theorem}
\begin{proof}
As in the proof of \cite[Theorem 12]{Gorenstein},
the differential form $\omega:=\frac{\phi(x,y)}{F_y} dx$
with separating variable $x$
is regular at every point of $\Spec(A)$ if and only if $\phi\in {\frak C}_A$,
and the differential form $\omega$ is written as
$-\frac{\phi'(x',y')}{F'_{y'}} dx'$ over the open affine part $U$ of $X\ne 0$,
where $\phi'$ is as in the statement of the theorem.
Let $F_0(X,Y,Z)$ be the homogenization of $F(x,y)$ (i.e., the defining equation of $C'$). We claim that $X\ne 0$ at any point of $C'\smallsetminus C$.
Indeed by the assumption that $F(x,y)$ contains the term $y^N$ with $N=\deg(F)$, if $Z=X=0$ held, then $Y=0$ has to hold. By this claim, it suffices to discuss
the regularity of $\omega$ over $U$.
This is equivalent to $\phi'\in {\frak C}_{A'}$,
since $x'$ is a separating variable of $k(C)$
and $F'(x',y')$ is a monic polynomial in $y'$ over $k[x']$ of degree $N$. The theorem clearly follows from these facts.
\end{proof}
\begin{remark}\label{rem:Gorenstein}
\begin{enumerate}
\item[(1)]
Here, $\phi(x,y)\in {\frak C}_A$
is equivalent to $\phi(x,y) \in {\frak C}_{A,P}=\{z\in A_P : z {\overline{A_P}} \subset A_P\}$
for any maximal ideal $P$ of $A$, where ${\overline{A_P}}$ is the integral closure of $A_P$ in $k(C)$.
Moreover ${\frak C}_{A,P} = {\frak C}_{A,P}^* \cap A_P$, where
${\frak C}_{A,P}^*$ is the conductor of $A_P^*$ in ${\overline{A_P}}^*$,
where $*$ means taking the completion (cf.\ \cite[Theorem 2]{Gorenstein}).
For example, if $C'$ has a single singular point, say $P$, in $C'\smallsetminus C$, then in order to see $\frac{\phi(x,y)}{F_y} dx \in H^0({\tilde C},\varOmega_{\tilde C})$ for $\phi \in {\frak C}_A$ it suffices to check $\phi' \in {\frak C}_{A',P}$ instead of checking $\phi' \in {\frak C}_{A'}$.
\item[(2)] As written in Theorem \ref{thm2:Gorenstein}, the degree of $\phi$ is at most $N-3$.
Similarly $\deg(\phi')\le N-3$ hold. 
Hence, the space $H^0({\tilde C}, \varOmega_{\tilde C})$ is the intersection of
${\frak C}_A^{(N-3)} := \{ \phi \in {\frak C}_A \mid \deg(\phi) \leq N-3\}$
and ${\frak C}_{A'}^{(N-3)}$,
where we identify $\phi \in {\frak C}_A^{(N-3)}$ and $\phi' \in {\frak C}_{A'}^{(N-3)}$ 
with the notation of Theorem \ref{thm2:Gorenstein}.
Since ${\frak C}_A^{(N-3)}$ and ${\frak C}_{A'}^{(N-3)}$
 are finite dimensional $k$-vector spaces,
it suffices to give an algorithm to describe $\frak C_A^{(N-3)}$,
in order to have an algorithm to describe $H^0({\tilde C}, \varOmega_{\tilde C})$.
\end{enumerate}
\end{remark}


\subsection{Complementary modules}\label{subsec:complementary}

As we have seen in Subsection \ref{subsec:Gorenstein}, regular deferential forms of $\tilde{C}$ are obtained from the conductor $\mathfrak{C}$ of $k[C]$.
In this subsection, we review the notion of {\it complementary modules}, which enables us to compute a basis of $\mathfrak{C}$.
Let us start with stating the definition of a complementary module of an integrally closed domain.

\begin{definition}[Complementary modules]\label{def:comp_mod}
Let $R$ be an integrally closed domain, and $K$ its field of fractions.
Let $K'$ be a finite separable (and thus simple) extension of $K$.
Let $\mathrm{Tr}_{K'/K} : K' \rightarrow K$ be the trace map for $K' / K$.
For an integral extension $\tilde{R} \subset K'$ of $R$, we call an $R$-module
\[
\mathcal{C}_{\tilde{R}/R} := \{ z \in K' : \mathrm{Tr}_{K'/K} (z \tilde{R}) \subset R \}
\]
the {\it complementary module} of $\tilde{R}$ with respect to $R$.
\end{definition}

Note in Definition \ref{def:comp_mod} that $\tilde{R} \subset \mathcal{C}_{\tilde{R}/R}$ since $\mathrm{Tr}_{K'/K}(r) \in R$ for all $r \in \tilde{R}$.
The next theorem is a consequence of \cite[Ch.~V, \S~11, Theorem 29]{ZS}.

\begin{theorem}[\cite{Mnuk}, Corollary 3.1]\label{cor:Mnuk}
Let the notation be same as in Definition \ref{def:comp_mod}.
Assume that $R$ is a Dedekind ring.
Let ${\overline R}$ be the integral closure of $R$ in $K'$.
Let $y \in {\overline R}$ be an element such that $K' =K(y)$, and $f(Y)$ be a minimal polynomial of $y$ over $K$.
Then we have
\[
\mathfrak{f}({\overline R}/R[y]) = f' (y) \mathcal{C}_{{\overline R}/R} ,
\]
where $\mathfrak{f}({\overline R}/R[y])$ is the conductor of $R[y]$ in ${\overline R}$, that is,
\begin{eqnarray}
\mathfrak{f}({\overline R}/R[y]) &=& \{ a \in R[y] : \forall b \in {\overline R},\ a \cdot (b \bmod{R[y]}) = 0 \bmod{R[y]} \} \nonumber \\
&=& \{ a \in R[y] : a {\overline R} \subset R[y] \}. \nonumber
\end{eqnarray}
\end{theorem}

Note that we can always take an element $y$ as in Theorem \ref{cor:Mnuk}, whihc is proved in a way similar to the first part of the proof of Lemma \ref{lem:PID} below.

In the following, we show that both $\overline{R}$ and $\mathcal{C}_{\overline{R}/R}$ in Theorem \ref{cor:Mnuk} are finite free $R$-modules of same rank if $R$ is a principal ideal domain, and that a basis of $\mathcal{C}_{\overline{R}/R}$ can be computed from that of $\overline{R}$. 
For this, we here recall the notion of dual bases in a finite separable field extension.
Proofs of the following lemmas would be well-known, but let us write all the complete proofs for readers' convenience:

\begin{lemma}\label{lem:dual_basis}
Let $K'/K$ be a finite separable field extension, and $\{ w_1, \ldots , w_N \}$ a basis of $K'$ over $K$.
Put $w_i^{\ast} := \sum_{j=1}^N c_{i,j} w_j$ with $c_{i,j} \in K$ for $1 \leq i \leq N$.
Then we have the following:
\begin{enumerate}
\item $(  \mathrm{Tr}_{K'/K}(w_i w_j) ) \in \mathrm{GL}_N(K)$.
\item By putting $(c_{i,j}) = (  \mathrm{Tr}_{K'/K}(w_i w_j) )^{-1}$, the basis $\{ w_1^{\ast}, \ldots , w_N^{\ast} \}$ of $K'$ over $K$ satisfies $\mathrm{Tr}_{K'/K}(w_i w_j^{\ast}) = \delta_{i,j}$ for each $1 \leq i,j \leq N$, where $\delta_{i,j}$ is the Kronecker delta.
Conversely, if $\mathrm{Tr}_{K'/K}(w_i w_j^{\ast}) = \delta_{i,j}$ for each $1 \leq i,j \leq N$, then $(c_{i,j}) = (  \mathrm{Tr}_{K'/K}(w_i w_j) )^{-1}$.
In this case, $\{ w_1^{\ast}, \ldots , w_N^{\ast} \}$ is the dual basis of $\{ w_1, \ldots , w_N \}$.
\end{enumerate}
\end{lemma}

\begin{proof}
(1) Let $\sigma_1, \ldots , \sigma_N$ be all $K$-embeddings from $K'$ to $\overline{K}$.
Putting $H = (\sigma_i(w_j))$, one can check that $( \mathrm{Tr}_{K'/K} (w_i w_j ) )_{i,j} = {}^t H H$, where ${}^t H$ is the transpose of $H$.
Thus, if $\mathrm{det}(  \mathrm{Tr}_{K'/K}(w_i w_j) )  = 0$, i.e., $\mathrm{det}(H) = 0$, then the row vectors of $H$ are linearly dependent over $K$, and hence there exists $(c_1, \ldots , c_N) \in K^N \smallsetminus \{ (0, \ldots , 0) \}$ such that $\sum_{j=1}^N c_j \sigma_j (w_i) = 0$.
Since $\{ w_1, \ldots , w_N \}$ generates $K'$ over $K$, we also have $\sum_{j=1}^N c_j \sigma_j (y) = 0$ for all $y \in K'$.
This means that $\sigma_1, \ldots , \sigma_N$ are linearly dependent over $K$, which contradicts Dedekind's lemma.

(2) It follows from the $K$-linearity of $\mathrm{Tr}_{K'/K}$ that
\[
\sum_{k=1}^N c_{i,k} \mathrm{Tr}_{K'/K} (w_k w_j ) =  \mathrm{Tr}_{K'/K} \left( \left(\sum_{k=1}^N c_{i,k} w_k \right) w_j \right)  = \mathrm{Tr}_{K'/K} ( w_i^{\ast} w_j) ,
\]
i.e., $(c_{ij}) \cdot (  \mathrm{Tr}_{K'/K}(w_i w_j) ) = (\mathrm{Tr}_{K'/K} ( w_i^{\ast} w_j))$, and thus the claim holds.
\end{proof}

\begin{lemma}\label{lem:dual_basis2}
Let the notation be same as in Definition \ref{def:comp_mod}, and suppose that $\tilde{R}$ is a free $R$-module of finite rank with a basis $\{ w_1, \ldots , w_N \}$ with $N:=[K':K]$.
Then we have the following:
\begin{enumerate}
\item $\{ w_1, \ldots , w_N \}$ is a basis of $K'$ over $K$.
\item The dual basis $\{ w_1^{\ast}, \ldots , w_N^{\ast} \}$ given in Lemma \ref{lem:dual_basis} generates $\mathcal{C}_{\tilde{R}/R}$ over $R$.
Hence $\mathcal{C}_{\tilde{R}/R}$ is a free $R$-module of rank $N$ with a basis $\mathcal{C}_{\tilde{R}/R}$.
\end{enumerate}
\end{lemma}

\begin{proof}
(1) Assuming $\sum_{i=1}^N b_i w_i = 0$ for $b_i \in K$, we first prove the $K$-linear independence of $\{ w_1, \ldots , w_N \}$.
Since $K$ is the field of fractions for $R$, we can write $b_i = p_i/q_i$ for some $p_i, q_i \in R$ with $q_i \neq 0$, and thus $\sum_{i=1}^N (p_i/q_i) w_i = 0$.
Multiplying the product $q_1 \cdots q_N$ to the both sides, we have $\sum_{i=1}^N \left( p_i \prod_{i \neq k} q_k \right) w_i = 0$.
It follows from the $R$-linear independence of $\{ w_1, \ldots , w_N \}$ that $ p_i \prod_{i \neq k} q_k  = 0$ for all $1 \leq i \leq N$. 
Since $q_k \neq 0$ for $1 \leq k \leq n$, one has $p_i = 0$ and thus $b_i = 0$.
Hence $\{ w_1, \ldots , w_N \}$ is linearly independent over $K$, and thus it follows from $[K':K]=N$ that $\{ w_1, \ldots , w_N \}$ spans $K'$ over $K$.

(2) For any element $r'=\sum_{i=1} b_i w_i \in \tilde{R}$ ($b_i \in R$), we have
\[
\mathrm{Tr}_{K'/K}(w_j^{\ast} r') = \sum_{i=1}^N b_i  \mathrm{Tr}_{K'/K} ( w_i w_j^{\ast} )
= \sum_{i=1}^N b_i  \delta_{j,i} = b_i \in R,
\]
and thus $w_j^{\ast} \in \mathcal{C}_{\tilde{R}/R}$ for each $1 \leq j \leq N$.
Next, any element $z' \in \mathcal{C}_{\tilde{R}/R} \subset K'$ is written as $z' = \sum_{j=1}^N a_{j} w_j^{\ast} \in K'$ ($a_j \in K$) since $\{ w_1^{\ast}, \ldots , w_N^{\ast} \}$ is a basis of $K'$ over $K$ by (1) together with Lemma \ref{lem:dual_basis} (2).
It suffices to show that we can take $a_j \in R$ for all $1 \leq j \leq N$.
For each $1 \leq i \leq N$, we have
\[
\mathrm{Tr}_{K'/K}(z' w_i) = \sum_{j=1}^N a_j  \mathrm{Tr}_{K'/K} ( w_i w_j^{\ast} )
= \sum_{j=1}^N a_j  \delta_{j,i} = a_i,
\]
and thus it follows from $z' \in \mathcal{C}_{\tilde{R}/R}$ and $w_i \in \tilde{R}$ that $a_i=\mathrm{Tr}_{K'/K}(z' w_i) \in R$, as desired.
\end{proof}

\begin{lemma}\label{lem:PID}
Let $R$ be a principal ideal domain (and thus integrally closed) and $K$ its field of fractions.
Let $K'$ be a finite separable extension of degree $N$ of $K$, and ${\overline R}$ the integral closure of $R$ in $K'$.
Then ${\overline R}$ is a free $R$-module of rank $N$.
\end{lemma}

\begin{proof}
Let $\{ w_1, \ldots , w_{N} \}$ be a basis of $K'$ over $K$.
We may assume $w_i \in {\overline R}$ for all $1 \leq i \leq N$.
Indeed, let $w \in K'$, and $w^{m} + (p_{m-1}/q_{m-1}) w^{m-1} + \cdots + (p_1/q_1) w + p_0/q_0$ its algebraic relation over $K$, where $p_i , q_i \in R$ with $q_i \neq 0$ for $0 \leq i \leq m-1$.
Then, multiplying $a^m$ with $a:=q_0 \cdots q_{m-1} \in R$ to the both sides, we have that $a w$ is integral over $R$, and thus $a w \in {\overline R}$.
Thus, there exists $a_i \in R \smallsetminus \{ 0 \}$ with $a_i w_i \in {\overline R}$ for each $1 \leq i \leq N$, and clearly $\{ a_1 w_1, \ldots , a_{N}w_{N} \}$ is also a basis of $K'$ over $K$.

Let ${\tilde R} \subset {\overline R}$ be the free $R$-module spanned by $\{ w_1, \ldots , w_{N} \}$.
Note that ${\tilde R}$ is integral over $R$, and that ${\tilde R} \subset {\overline R} \subset \mathcal{C}_{{\overline R}/R} \subset \mathcal{C}_{{\tilde R}/R}$ in $K'$.
By Lemma \ref{lem:dual_basis2} (2), $\mathcal{C}_{{\tilde R}/R}$ is a free $R$-module of rank $N$.
Since $R$ is a principal ideal domain, both ${\overline R}$ and $\mathcal{C}_{{\overline R}/R}$ are also free $R$-modules, and have rank $N$.
\end{proof}

\section{Computing syzygies and normalization}\label{sec:normalization}

In this section, we recall the notion of Gr\"{o}bner bases of free modules briefly, and review methods to compute syzygies of polynomials and the normalization of an affine ring.

\subsection{Gr\"{o}bner bases of free modules}\label{subsec:Groebner}
Let $S = k[x_1, \ldots , x_n]$ be the polynomial ring of $n$ variables over a field $k$.
For readers' convenience, we start with recalling the notion of Gr\"{o}bner bases in the polynomial ring $S$, and its extension to free $S$-modules.

A {\it monomial (or term) order} on $S$ is a total order $\prec$ on the set of monomials $\mathrm{Mon}(S) := \{ x_1^{\alpha_1} \cdots x_n^{\alpha_n} : ( \alpha_1, \ldots , \alpha_n ) \in (\mathbb{Z}_{\geq 0})^n \}$ such that (1) $m_1 \preceq m_2$ for $m_1,m_2 \in {\rm Mon}(S)$ implies $m_1 m_3 \preceq m_2 m_3 $ for all $m_3 \in {\rm Mon}(S)$, and (2) $1 \preceq m$ for all $m \in {\rm Mon}(S)$.
For $f \in S \smallsetminus \{ 0 \}$, we denote by $\mathrm{LT}_{\prec}(f)$, $\mathrm{LM}_{\prec}(f)$ and $\mathrm{LC}_{\prec}(f)$ the leading term, the leading monomial, and the leading coefficient of $f$ with respect to $\prec$, respectively.
Note that $\mathrm{LT}_{\prec}(f) = \mathrm{LC}_{\prec}(f) \cdot \mathrm{LM}_{\prec}(f)$.
For a subset $F \subset S$, we set $\mathrm{LT}_{\prec}(F) := \{\mathrm{LT}_{\prec}(f) : f  \in F \}$ and $\mathrm{LM}_{\prec}(F) := \{ \mathrm{LM}_{\prec}(f) : f \in F \}$.
A finite subset $G \subset S$ is called a {\it Gr\"{o}bner basis} for an ideal $I \subset S$ with respect to $\prec$ if it generates $I$ and if $ \langle \mathrm{LT}_{\prec}(G) \rangle_S = \langle \mathrm{LT}_{\prec}(I) \rangle_S$.
For simplicity, we denote $\mathrm{LT}_{\prec}$ as $\mathrm{LT}$ and so on, if $\prec$ is clear from the context.
It is well-known that every non-zero ideal of $S$ has a Gr\"{o}bner basis with respect to an arbitrary monomial order, and so far various algorithms for computing Gr\"{o}bner bases have been proposed, e.g., Buchberger's algorithm~\cite{Buchberger}, Faug\`{e}re's $F_4$~\cite{F4} and $F_5$~\cite{F5}, and so on.

The notion of monomial orders and Gr\"{o}bner bases is extended to free modules $S^r = \bigoplus_{i=1}^r S e_i$, where $e_i = (0, \ldots , 0 , 1 , 0 , \ldots, 0)$ denotes the $i$-th standard basis vector of $S^r$ with $1$ at the $i$-th coordinate and $0$'s elsewhere.
A monomial in $S^r$ is an element of the form $m e_i = (0, \ldots , 0 , m, 0, \ldots , 0) \in S^r$ with $m \in \mathrm{Mon}(S)$, and the set of all monomials in $S^r$ is denoted by $\mathrm{Mon}(S^r)$.
For two monomials $m e_i$ and $m' e_j$ with $m,m' \in \mathrm{Mon}(S)$, we say that $m e_i$ is {\it divisible} by $m' e_j$ if $i=j$ and $m' \mid m$.
A (module) monomial order on $S^r$ is defined as a total order $\prec$ (we use the same notation as in the case of polynomial rings) on $\mathrm{Mon}(S^r)$ such that (1)' $m_1 \preceq m_2$ for $m_1,m_2 \in {\rm Mon}(S^r)$ implies $m m_1 \preceq m m_2 $ for all $m \in {\rm Mon}(S)$ (!), and (2)' $e_i$ is the minimum element in the set of all monomials in $S e_i$, namely $e_i \preceq m e_i$ for all $m \in {\rm Mon}(S)$.
For an element $f \in S^r$, its leading term, leading monomial, and leading coefficient with respect to $\prec$ are defined similarly to the case of an element in $S$, and the same notations $\mathrm{LT}_{\prec}(f)$, $\mathrm{LM}_{\prec}(f)$ and $\mathrm{LC}_{\prec}(f)$ are used.
As for ideals, we also define a Gr\"{o}bner basis for a {\it submodule} $N \subset S^r$:
A finite subset $G \subset N$ is called a Gr\"{o}bner basis for $N$ with respect to $\prec$ if $ \langle \mathrm{LT}_{\prec}(G) \rangle_S = \langle \mathrm{LT}_{\prec}(I) \rangle_S$, equivalently, for every $f \in I$, there exists $g \in G$ such that $\mathrm{LT}_{\prec}(f)$ is divisible by $\mathrm{LT}_{\prec}(g)$, where we denote by $\langle F \rangle_S$ the $S$-submodule of $S^r$ generated by a subset $F \subset S^r$.
A Gr\"{o}bner basis for a submodule $N \subset S^r$ with respect to an arbitrary module monomial order always exists as in the case of ideals, and it can be computed by extensions of algorithms such as \cite{Buchberger}, \cite{F4} and \cite{F5} over polynomial rings to free modules.
In fact, implementations of such extended algorithms are found in e.g., Macaulay2, Magma and Singular.


In the computation of Gr\"{o}bner bases in free modules, there are two typical monomial orders: {\it POT} and {\it TOP}, where POT (resp.\ TOP) stands for ``position-over-term'' (resp.\ ``term-over-position'').
These orders will be used in the syzygy computation of our main algorithm in Section \ref{sec:main} below.
The definition of POT and TOP orders is as follows:
For a monomial order $\succ$ on $\mathrm{Mon}(S)$, we say
\begin{itemize}
\item $m e_i \succ_{\rm POT} m' e_j$ with $m, m' \in \mathrm{Mon}(S)$ if $i < j$, or if $i=j$ and $m \succ m'$.
\item $m e_i \succ_{\rm TOP} m' e_j$ with $m, m' \in \mathrm{Mon}(S)$ if $m \succ m'$, or if $m = m'$ and $i<j$.
\end{itemize}
These two relations are exactly monomial orders on $\mathrm{Mon} (S^r)$, and $\succ_{\rm POT}$ (resp.\ $\succ_{\rm TOP}$) is called a POT (resp.\ TOP) extension of $\succ$.
Note that we have $e_1 \succ \cdots \succ e_r$ in both two cases, but other orders such as $e_1 \prec \cdots \prec e_r$ are also possible;
in such a case, we specify the order of $e_i$'s. 

We will also use a {\it product} (or {\it block}) order for $S^r$, like an elimination order for polynomial rings:
While an elimination order for $S$ divides the variables $x_1, \ldots , x_n$ into two (or more) blocks, this order divides the monomials $e_1, \ldots , e_r$ into two (or more) blocks.
For two module monomial orders $\succ_1$ on $\bigoplus_{i=1}^k S e_i$ and $\succ_2$ on $\bigoplus_{i=k+1}^r S e_i$, their product order $\succ_{1,2} \; = (\succ_1, \succ_2)$ on $S^r = \bigoplus_{i=1}^r S e_i$ is defined as follows:
For two monomials $m e_i$ and $m' e_j$ with $m,m' \in \mathrm{Mon}(S)$, we say $m e_i \succ_{1,2} m' e_j$ if $i, j \leq k$ and $m e_i \succ_1 m' e_j$, $i \leq k$ and $k < j$, or $k < i, j$ and $m e_i \succ_2 m' e_j$.
The order $\succ_{1,2}$ has the following {\it elimination property} with respect to $e_1, \ldots , e_k$:
$f \in S^r$ with $\mathrm{LT}_{\succ_{1,2}}(f) \in \bigoplus_{i=k+1}^r S e_i$ satisfies $f \in \bigoplus_{i=k+1}^r S e_i$.
The definition of product order is clearly extended to arbitrary number of blocks, and the POT order is a particular case of product orders with the elimination property with respect to $e_1, \ldots , e_k$ for each $1 \leq k \leq r-1$.

\subsection{Computing syzygy modules}\label{subsec:syzygies}

This subsection reviews methods for computing syzygy modules.
We use the same notation as in the previous subsection.

Let $f_1, \ldots , f_r \in S = k[x_1, \ldots , x_n]$.
Consider the following $S$-submodule of $S^r$:
\begin{equation}\label{eq:syzygy}
\mathrm{syz}(f_1, \ldots , f_r) := \left\{ (a_1, \ldots , a_r) \in S^r = \bigoplus_{i=1}^r S e_i : \sum_{i=1}^r a_i f_i = 0 \right\},
\end{equation}
which is nothing but the kernel of the $S$-homomorphism $S^r \rightarrow S$ defined by $\ e_i \mapsto f_i$ for $1 \leq i \leq r$.
We call $\mathrm{syz}(f_1, \ldots , f_r )$ the {\it module of syzygies} (or {\it syzygy module}) of $(f_1, \ldots , f_r)$, and each element in $\mathrm{syz}(f_1, \ldots , f_r)$ a {\it syzygy} of $(f_1, \ldots , f_r)$.
Note that $\mathrm{syz}(f_1, \ldots , f_r)$ is not the zero module.
Indeed, if $f_i = 0$ for some $1 \leq i \leq r$ then $e_i \in \mathrm{syz} (f_1, \ldots , f_r)$, and if $f_i \neq 0$ for all $1 \leq i \leq r$ then $\mathrm{syz}(f_1, \ldots , f_r)$ contains a non-zero element such as $(f_r, 0 , \ldots , 0, -f_1)$.

The following lemma shows that given $f_1, \ldots, f_r$, a Gr\"{o}bner basis of $\mathrm{syz}(f_1, \ldots , f_r)$ can be obtained by computing a Gr\"{o}bner basis for a submodule in $S^{r+1}$:

\begin{lemma}\label{lem:ComputeSyzygy}
Let $f_1, \ldots , f_r \in S $, and let $F = \langle f_1 e_0 + e_1, \ldots , f_r e_0 + e_r \rangle_S$ be the submodule of $S^{r+1} = \bigoplus_{i=0}^r Se_i$ generated by
\begin{eqnarray}
f_1 e_0 + e_1 &=& (f_1, 1, 0 , 0, \ldots , 0, 0), \nonumber \\
f_2 e_0 + e_2 &=& (f_2, 0, 1 , 0, \ldots , 0, 0), \nonumber \\
&\vdots & \nonumber \\
f_r e_0 + e_r &=& (f_r, 0, 0 , 0, \ldots , 0, 1), \nonumber 
\end{eqnarray}
where $\{ e_0, e_1, \ldots , e_r \}$ is the standard basis of $S^{r+1}$.
Let $G$ be a Gr\"{o}bner basis of $F$ with respect to a monomial order $\succ$ with the elimination property with respect to $e_0$ (for example, a POT order with $e_0 \succ e_i$ for all $1 \leq i \leq r$).
Then, $G$ contains an element of the form $(0, g_1, \ldots , g_r ) \in S^{r+1}$ with $(g_1, \ldots , g_r)\neq  (0, \ldots , 0 ) $.

Moreover, if $G$ is a Gr\"{o}bner basis of $F$ with respect to $\succ$, then the set 
\begin{eqnarray}\label{eq:syzGB}
\{ (g_1, \ldots , g_r) : (0, g_1, \ldots , g_r ) \in G \} \subset \mathrm{syz}(f_1, \ldots , f_r)
\end{eqnarray}
is also a Gr\"{o}bner basis of $\mathrm{syz}(f_1, \ldots , f_r) \subset S^r = \bigoplus_{i=1}^r S e_i $ with respect to the monomial order on $S^r$ induced from the restriction of $\succ$ on $S^{r+1}$.
In particular, the set \eqref{eq:syzGB} generates $\mathrm{syz}(f_1, \ldots , f_r)$.
\end{lemma}

\begin{proof}
Let $(h_1, \ldots , h_r) \in \mathrm{syz}(f_1, \ldots , f_r) \smallsetminus \{ (0, \ldots , 0) \}$.
Then we have
\[
h:= (0, h_1, \ldots , h_r ) = \left( \sum_{i=1}^r h_i f_i , h_1, \ldots , h_r \right) = \sum_{i=0}^r h_i (f_i e_0 + e_i ) \in F \smallsetminus \{ (0, \ldots , 0 ) \},
\]
and thus $\mathrm{LM}(h) $ is divisible by $\mathrm{LM}(g)$ for some $g \in G$.
Since $\mathrm{LM}(g) \in \bigoplus_{i=1}^r S e_i $, it follows from the elimination property of $\succ$ that $g \in \bigoplus_{i=1}^r S e_i$, i.e., $g$ is of the form $(0, g_1, \ldots , g_r )$.
This also implies that \eqref{eq:syzGB} is a Gr\"{o}bner basis of $\mathrm{syz}(f_1, \ldots , f_r)$.
\end{proof}

Computing a Gr\"{o}bner basis of a syzygy module also gives a solution to the following problem:
Given $f, f_1, \ldots f_r \in S = K[x_1, \ldots , x_n]$, compute $h_i \in S$ such that $f = \sum_{i=1}^r h_i f_i$, when we know $f \in \langle f_1, \ldots , f_r \rangle_S$.
The following lemma provides a solution:

\begin{lemma}\label{lem:syz}
Let $f, f_1, \ldots , f_r \in S = k[x_1,\ldots , x_n]$.
If $f = \sum_{i=1}^r h_i f_i$ for $h_i \in S$ with $1 \leq i \leq r$, then $(1, - h_1, \ldots , -h_r) \in \mathrm{syz}(f,f_1, \ldots , f_r) \subset S^{r+1} = \bigoplus_{i=0}^r S e_i$.
Moreover, for any monomial order $\succ$ on $S^{r+1}$ with the elimination property with respect to $e_0$ (for example, a POT order with $e_0 \succ e_i$ for all $1 \leq i \leq r$), a Gr\"{o}bner basis of $\mathrm{syz}(f,f_1, \ldots , f_r)$ with respect to $\succ$ includes an element of the form $(c, g_1, \ldots , g_r) \in S^{r+1}$ with $c \in k^{\times}$ and $g_i \in S$ for $1 \leq i \leq r$.
Thus, putting $h_i := - c^{-1} g_i$ with $1 \leq i \leq r$ for such an element $(c, g_1, \ldots , g_r) $, we have $(1, - h_1, \ldots , -h_r) \in \mathrm{syz}(f,f_1, \ldots , f_r) $, so that $f = \sum_{i=1}^r h_i f_i$.
\end{lemma}

\begin{proof}
The former claim is straightforward.
We show the latter claim.
Putting $N:=\mathrm{syz}(f,f_1, \ldots , f_r)$ and $h:=(1,-h_1,\ldots,-h_r) \in N$, it follows from the elimination property of $\succ$ that $\mathrm{LT}(h) = e_1 \in \mathrm{LT}(N) $.
Thus, letting $G$ be a Gr\"{o}bner basis for $N$, there exists $g \in G$ such that $e_1 = \mathrm{LT}(h)$ is divisible by $\mathrm{LT}(g)$.
Hence $\mathrm{LT}(g) = c e_1$ for some $c \in k^{\times}$, and thus the first entry of $g$ is $c$ by the definition of module monomial orders.
\end{proof}

The notion of syzygies is extended for quotient rings.
More precisely, for polynomials $f_1, \ldots , f_r \in S$ and an ideal $I \subset S$, we consider the following two modules:
\begin{equation}\label{eq:modulo-syzygy}
\mathrm{syz}_I(f_1, \ldots , f_r) := \left\{ (a_1, \ldots , a_r) \in \bigoplus_{i=1}^r S e_i : \sum_{i=1}^r a_if_i = 0 \bmod{I} \right\}
\end{equation}
and
\begin{equation}\label{eq:modulo-syzygy2}
\mathrm{syz}_I (\overline{f_1}, \ldots , \overline{f_r}) := \left\{ (\overline{a}_1, \ldots , \overline{a}_r) \in \bigoplus_{i=1}^r (S/I) e_i : \sum_{i=1}^r a_i f_i  = 0 \bmod{I} \right\},
\end{equation}
where $\overline{f}$ denotes $f \bmod{I}$ for $f \in S$.
Given $f_1, \ldots , f_r$ and a generator set of $I$, we can compute generator sets of these ``modulo'' syzygies by the following lemma:

\begin{lemma}\label{lem:ModSyz}
Let $f_1, \ldots , f_r, f_{r+1}, \ldots , f_{r+s} \in S$, and $I = \langle f_{r+1}, \ldots , f_{r+s} \rangle_S$ an ideal generated by $f_{r+1}, \ldots , f_{r+s}$. 
For an element $f \in S$, we denote $f \bmod{I}$ by $\overline{f}$.
Let $\varphi$ be an $S$-homomorphism from $S^r = \bigoplus_{i=1}^r S e_i$ to $S$ defined by $e_i \mapsto f_i$ for $1 \leq i \leq r$.
Extending $\varphi$ to $S^{r+s} = \bigoplus_{i=1}^{r+s} S e_i $, say $\psi : \bigoplus_{i=1}^{r+s} S e_i \rightarrow S \ ; \ e_i \mapsto f_i$, we have the commutative diagram:
\[
\xymatrix{
\bigoplus_{i=1}^{r+s} S e_i \ar[dr]^{\psi} \ar[d]_{\rm pr} & &  \\
\bigoplus_{i=1}^{r} S e_i \ar[d]_{\bmod{I}} \ar[r]^{\varphi} & S \ar[r]^{\bmod{I}} & S/I ,\\
\bigoplus_{i=1}^{r} (S/I) e_i \ar[urr]^{\tilde{\varphi}} & &
}
\]
where ``$\mathrm{pr}$'' is the canonical projection from $\bigoplus_{i=1}^{r+s} S e_i$ onto $\bigoplus_{i=1}^r S e_i$.
Let $\overline{\varphi}$ be the composite map of $\varphi$ and $\bmod{\; I}$, and $\mathcal{G}$ a generator set of $\mathrm{Ker}(\psi)$.

Then $\mathrm{pr}(\mathcal{G})$ and $\overline{\mathrm{pr}(\mathcal{G})}$ generate $\mathrm{Ker}(\overline{\varphi})$ and $\mathrm{Ker}(\tilde{\varphi})$ respectively, where $\overline{\mathrm{pr}(\mathcal{G})}$ denotes the image of $\mathrm{pr}(\mathcal{G})$ in $(S/I)^r =  \bigoplus_{i=1}^{r} (S/I) e_i$.
Moreover, if $\mathcal{G}$ is a Gr\"{o}bner basis for $\mathrm{Ker}(\psi)$ with respect to a monomial order $\succ$ on $S^{r+s}$ with the elimination property with respect to $e_1, \ldots , e_r$ (for example, a POT order with $e_i \succ e_j$ ($1 \leq i \leq r$, $r+1 \leq j \leq r+s$), then $\mathrm{pr}(\mathcal{G})$ is also a Gr\"{o}bner basis for $\mathrm{Ker}(\overline{\varphi})$ with respect to the monomial order $\succ'$ on $S^r$ induced from the restriction of $\succ$.
\end{lemma}

\begin{proof}
Letting $\mathcal{G} = \{ (h_{1,k}, \ldots , h_{r,k}, h_{r+1,k}, \ldots , h_{r+s,k} ) : 1 \leq k \leq m \}$ be a generator set of $\mathrm{Ker}(\psi)$, we have $\mathrm{pr}(\mathcal{G}) = \{ {h}_k := (h_{1,k}, \ldots ,h_{r,k} ) : 1 \leq k \leq m \}$.
Since $\mathrm{pr}(\mathcal{G}) \subset \mathrm{Ker}(\overline{\varphi})$ is clear, we here prove the converse $\langle \mathrm{pr}(\mathcal{G}) \rangle_S \supset \mathrm{Ker}(\overline{\varphi})$.
Let $h \in \mathrm{Ker}(\overline{\varphi})$ with $h= (h_1, \ldots , h_r) \in \bigoplus_{i=1}^{r} S e_i$ ($h_i \in S$).
Since $h_1 f_1 + \cdots + h_r f_r \in I$, there exist $h_{r+1}, \ldots , h_{r+s} \in S$ such that $h_1 f_1 + \cdots + h_r f_r - h_{r+1} f_{r+1} - \cdots - h_{r+s} f_{r+s} = 0$.
Thus $\tilde{h}:=(h_1, \ldots , h_{r} , -h_{r+1}, \ldots , -h_{r+s}) \in \mathrm{Ker}(\psi)$, and so there exist $a_k \in S$ with $1 \leq k \leq m$ such that 
\[
(h_1, \ldots , h_{r} , -h_{r+1}, \ldots , -h_{r+s}) = \sum_{k=1}^m a_k \cdot (h_{1,k}, \ldots , h_{r,k}, h_{r+1,k}, \ldots , h_{r+s,k} ).
\]
Hence $(h_1, \ldots , h_{r} ) = \sum_{k=1}^m a_k \cdot (h_{1,k}, \ldots , h_{r,k})$, so that $h = \sum_{k=1}^m a_k h_k$.
We also have $\langle \overline{\mathrm{pr}(\mathcal{G})} \rangle = \mathrm{Ker}(\tilde{\varphi})$ since $\overline{\mathrm{Ker}(\overline{\varphi})} = \mathrm{Ker}(\tilde{\varphi})$.

Suppose that $\mathcal{G}$ is a Gr\"{o}bner basis for $\mathrm{Ker}(\psi)$ with respect to $\succ$.
In this case, for above $h$ and $\tilde{h}$ with $h \neq 0$, there exists $g \in \mathcal{G}$ such that $\mathrm{LT}_{\succ}(\tilde{h})$ is divisible by $\mathrm{LT}_{\succ}(g)$. 
It follows from the elimination property of $\succ$ that $\mathrm{LT}_{\succ}(\tilde{h}) = \mathrm{LT}_{\succ'}({h})$, and thus $\mathrm{LT}_{\succ}(g) \in \bigoplus_{i=1}^r S e_i$ and $\mathrm{LT}_{\succ}(g) = \mathrm{LT}_{\succ'}(\mathrm{pr} (g))$.
Thus $\mathrm{pr}(\mathcal{G})$ is also a Gr\"{o}bner basis for $\mathrm{Ker}(\overline{\varphi})$.
\end{proof}

\subsection{Computing the normalization as an affine ring}

Assume that $k$ is a perfect field.
Let $A$ be an affine domain, say $A = S / I$ for some prime ideal $I \subset S = k[x_1, \ldots , x_n]$, and let $\overline{A}$ denote the integral closure of $A$ in the field of fractions $Q(A)$.
It is well-known that $\overline{A}$ is also an affine domain (cf.\ \cite[Ch.\ V, Theorem 9]{ZS}).
In this subsection, we review a (generic) method given in \cite[\S 3.6]{GP} to compute the normalization $\overline{A}$ of $A$ as an affine domain.

The following lemma is straightforward:

\begin{lemma}[cf.\ {\cite[Lemma 3.6.1]{GP}}]\label{lem:int3}
Let $A$ be a Noether reduced ring, and $\tilde{J} \subset A$ an ideal containing a non-zerodivisor $a$.
Then we have the following:
\begin{enumerate}
\item The endomorphism ring $\mathrm{Hom}_A (\tilde{J},\tilde{J})$ is a commutative ring.
\item $A$ is a subring of $\mathrm{Hom}_A (\tilde{J}, \tilde{J})$ via the map which maps $x \in A$ to the multiplication by $x$. 
\item $\mathrm{Hom}_A(\tilde{J},\tilde{J}) \cong (1/a) \cdot (a \tilde{J} : \tilde{J}) \subset Q(A)$, where
\[
a \tilde{J} : \tilde{J} := \{ b \in A : b \tilde{J} \subset a \tilde{J} \}.
\]
\item $(1/a) \cdot (a \tilde{J} : \tilde{J}) \subset \overline{A}$.
\end{enumerate}
\end{lemma}

Here is a well-known criterion for normality by Grauert-Remmert~\cite{GR}:

\begin{proposition}[{\cite[pp.\ 220--221]{GR}}]\label{prop:criteria}
Let $A$ be a Noether reduced ring, and $\tilde{J}$ an ideal of $A$ satisfying the following three conditions:
\begin{enumerate}
\item $\tilde{J}$ contains a non-zerodivisor of $A$,
\item $\tilde{J}$ is a radical ideal, and
\item $N(A) \subset V(\tilde{J})$, where
\[
N(A) = \{ \mathfrak{p} \in \mathrm{Spec}(A) : \mbox{$A_{\mathfrak{p}}$ is not normal} \}
\]
is the non-normal locus of $A$.
\end{enumerate}
Then, $A$ is normal if and only if $A = \mathrm{Hom}_A(\tilde{J},\tilde{J})$.
\end{proposition}

The following lemma shows that the normality of $A$, i.e., $A = \mathrm{Hom}_A(\tilde{J},\tilde{J})$ is determined by the equality of ideals in $A$, and that the structure of $\mathrm{Hom}_A(\tilde{J},\tilde{J})$ as an affine ring is also determined:

\begin{lemma}[{\cite[Lemma 3.6.7]{GP}}]\label{lem:normal}
Let $A$ be a Noether reduced ring, and $\tilde{J} \subset A$ an ideal containing a non-zerodivisor $a$.
Then we have the following:
\begin{enumerate}
\item $A = \mathrm{Hom}_A(\tilde{J},\tilde{J})$ is equivalent to $a \tilde{J} : \tilde{J} = \langle a \rangle_A$.
\end{enumerate}
Let $\{ u_0 = a, u_1, \ldots , u_s \}$ be a generator set of the ideal $a \tilde{J} : \tilde{J}$.
\begin{enumerate}
\item[(2)] For each $0 \leq i, j \leq s$, there exist $\xi_{k}^{ij} \in A$ with $u_i \cdot u_j = \sum_{k=0}^s a \xi_{k}^{ij} u_k$ in $A$. 
\end{enumerate}
Let $\mathrm{syz}(u_0, \ldots , u_s)$ be the syzygy module of $(a \tilde{J} : \tilde{J})$, that is, the kernel of the $A$-homomorphism $A^{s+1} \rightarrow (a \tilde{J} : \tilde{J}) \ ; \ e_j \mapsto u_j$, where $e_j$ denotes the vector in $A^{s+1}$ with $1$ in its $j$-th coordinate and $0$'s elsewhere.
Note that $\mathrm{syz}(u_0, \ldots , u_s)$ is finitely generated over $A$ since $A$ is Noetherian.
Let $\{ (\eta_0^{(k)}, \ldots , \eta_s^{(k)} ) : 1 \leq k \leq m \} \subset A^{s+1}$ be a generator set of $\mathrm{syz}(u_0, \ldots , u_s)$, and $\tilde{I} \subset A[t_1, \ldots , t_s]$ the ideal generated by
\begin{equation}\label{eq:GenSet}
\left\{ t_i t_j - \sum_{k=0}^s \xi_k^{ij} t_k : 1 \leq i \leq j \leq s \right\} \cup \left\{ \sum_{\nu=0}^s \eta_{\nu}^{(k)} t_{\nu} : 1 \leq k \leq m \right\},
\end{equation}
where $t_0 := 1$.
Then,
\begin{enumerate}
\item[(3)] $t_i \mapsto u_i / a$ ($1 \leq i \leq s$) defines a surjective ring homomorphism
\[
\phi : A[t_1, \ldots , t_s] \rightarrow \mathrm{Hom}_A( \tilde{J}, \tilde{J}) \cong (1/a) \cdot (a \tilde{J} : \tilde{J})
\]
with $\mathrm{Ker}(\phi) = \tilde{I}$, and thus $A[t_1, \ldots , t_s] /\tilde{I} \cong \mathrm{Hom}_A (\tilde{J},\tilde{J}) $ as $A$-algebras.
\end{enumerate}
\end{lemma}

Based on Lemma \ref{lem:int3}, Proposition \ref{prop:criteria} and Lemma \ref{lem:normal}, we obtain an algorithm to compute the normalization:

\begin{algorithm}[Sketch of {\cite[Algorithm 3.6.9]{GP}}]\label{alg:normalization}
~
\begin{enumerate}
\item[{\bf Input}:] A pair $(S,I)$ of a polynomial ring $S$ over $k$ and its prime ideal $I$.
\item[{\bf Output}:] A pair $(S',I')$ of a polynomial ring $S'$ over $k$ and its prime ideal $I'$.
\end{enumerate}
\begin{enumerate}
\item Choose a test pair $(\tilde{J}, \overline{a} )$ of $A := S/I$, where $\overline{a}$ denotes the residue class of an element $a \in S$ modulo $I$.
\item Test whether $\overline{a} \tilde{J} : \tilde{J} = \langle \overline{a} \rangle_A$ or not, by computing a generator set of $\overline{a} \tilde{J} : \tilde{J}$.
If $\overline{a} \tilde{J} : \tilde{J} = \langle \overline{a} \rangle_A$, then output $(S,I)$, otherwise go to the next step.
\item Compute a generator set of the ideal $I'$ in $S' = S[t_1, \ldots , t_s]$ corresponding to the ideal $\tilde{I}$ in $A[t_1, \ldots , t_s]$ generated by \eqref{eq:GenSet}.
Replace $(S,I)$ by $(S',I')$, and go back to (1). 
\end{enumerate}
\end{algorithm}

As for the computation of $\overline{a} \tilde{J} : \tilde{J} $, we use the following lemma (without proof):

\begin{lemma}\label{lem:quo_ideal}
Let $S$ be a ring, and $I \subset J \subset S$ ideals.
Put $A = S/I$ and $\tilde{J} = J/I$.
Then we have
\[
( (a J + I) : J) / I = \overline{a} \tilde{J} : \tilde{J} 
\]
in $A$ for all $a \in S$, where $\overline{a}$ denotes $a \bmod{I}$.
Note that $I \subset (a J + I) : J$ and $\langle a \rangle_S \subset (a J + I ): J$.

Hence $  \overline{a} \tilde{J} : \tilde{J}  = \langle \overline{a} \rangle_A$ is equivalent to $(a J + I ) : J \subset \langle a \rangle_S + I$.
\end{lemma}

In the following, we briefly explain details of each step of Algorithm \ref{alg:normalization}.
Let $\{ f_1, \ldots , f_{\ell} \}$ be a generator set of $I$ in the loop we consider.

For (1), a typical choice of $\tilde{J}$ is (the radical of) the Jacobian ideal of $I$, that is, the ideal in $A$ generated by the images of the $(n-r) \times (n-r)$-minors of the Jacobian matrix $(\partial f_i / \partial x_j)$, where $r = \mathrm{dim}(I)$ (cf.\ \cite[Section 5.7]{GP}).
Note that $\tilde{J} = A$ is equivalent to $A = \overline{A}$, by Lemma \eqref{lem:normal} (1).

To decide whether $\overline{a} \tilde{J} : \tilde{J} = \langle \overline{a} \rangle_A$ or not in (2), it suffices to alternatively check $(a J + I ): J \subset \langle a \rangle_S + I$, see Lemma \ref{lem:quo_ideal}.
For this, we compute a Gr\"{o}bner basis $\{u_0 = a,\ u_1, \ldots , u_s\} \subset S$ for $(a J + I ): J$, and check $u_i \in \langle a \rangle_S + I$ or not for each $0 \leq i \leq s$.
Note that $\{ \overline{u_i} := u_i \bmod{I} : 0 \leq i \leq s \}$ generates $\overline{a} \tilde{J} : \tilde{J}$.

In (3), we first compute $\xi_{k}^{ij} \in S$ such that $u_i \cdot u_j = \sum_{k=0}^s \xi_{k}^{ij} (a u_k ) + \sum_{k=1}^{\ell} \xi_{s+k}^{ij} f_k$ (cf.\ Lemma \ref{lem:normal}).
It follows from Lemma \ref{lem:syz} that this representation can be obtained by computing a Gr\"{o}bner basis for $\mathrm{syz}(u_i u_j, a u_0, \ldots , a u_s, f_1 , \ldots , f_{\ell})$ in $S^{s+\ell+2} = \bigoplus_{i=0}^{s+\ell+1} S e_i$ with respect to a monomial order $\succ$ on $S^{s+\ell+2}$ with the elimination property with respect to $e_0$.
The second step of (3) is to compute the set of representatives $\{ ( \eta_0^{(1)}, \ldots , \eta_s^{(1)}), \ldots , ( \eta_0^{(m)}, \ldots , \eta_s^{(m)}) \} \subset S^{s+1}$ of generators for $\mathrm{syz}_I(\overline{u_0}, \ldots , \overline{u_s}) \subset A^{s+1}$ with $A = S/I$.
For this, it suffices to compute a Gr\"{o}bner basis of $\mathrm{syz}(u_0, u_1, \ldots , u_s, f_1, \ldots , f_{\ell}) \subset S^{s+\ell +1} =\bigoplus_{i=0}^{s+\ell} S e_i$ with respect to an arbitrary monomial order on $S^{s+\ell + 1}$, see Lemma \ref{lem:ModSyz} for details.

\begin{example}\label{ex:normalization}
Consider a plane singular curve $C = V(F)$ in $\mathbb{A}^2(\overline{k})$ over $k:=\mathbb{F}_{11}$ defined by $F := X^5 + Y^5 + X Y$.
The desingularization of the Zariski closure $C' = V(X^5 + Y^5 + XYZ^3)$ of $C$ in $\mathbb{P}^2$ is known to be a superspecial curve of genus $5$ over $\mathbb{F}_{11}$, see \cite[Proposition 5.1.1]{trigonal} for details.
In the following, let us compute an integral closure of $k[C] = k[X,Y]/\langle F \rangle$ of $C$ by Algorithm \ref{alg:normalization} for the inputs $S = k[X,Y]$ and $I = \langle F \rangle$:

\noindent \underline{\bf 1st loop}: 
\begin{enumerate}
\item Take the radical of the Jacobian ideal of $I $ as $J$.
We then have $I + \mathrm{Jac}(I) = \langle X^5 + Y^5 + X Y, 5 Y^4 + X, Y + 5 X^4 \rangle $, and a Gr\"{o}bner basis of $J:=\sqrt{ I + \mathrm{Jac}(I) }$ with respect to the lexicographical order with $Y \succ X$ is $\{ Y, X \}$.
Thus we also take $a:=Y \in J \smallsetminus I$.

\item Computing a generator set of $(a J + I): J$, we have $(a J + I): J = \langle Y, X^4\rangle $, which is not equal to $\langle a \rangle_S + I$.
Put $u_0 := a=Y$ and $u_1 := X^4$.

\item We first compute $\xi_0^{11}$, $\xi_{1}^{11} \in S$ such that $u_1^2 = X^8 \equiv \xi_{0}^{11} (a u_0) + \xi_{1}^{11} (a u_1) \pmod{I}$.
Note that $a u_0 = a^2 = Y^2$ and $a u_1 = Y X^4$.
A Gr\"{o}bner basis for $\mathrm{syz}(u_1^2, a u_0, a u_1, F)$ in $\bigoplus_{i=0}^3 S e_i $ with respect to the POT order with $e_0 \succ e_1 \succ e_2 \succ e_3$ is computed as
\[
\{ ( 1, Y^3 X^3, 1, 10 X^3),\ ( 0, Y^4 + X, X, 10 Y),\ (0, X^4, 10 Y, 0),\ (0,0,F,10 Y X^4) \},
\]
from the first element of which we obtain a relation
\[
u_1^2 = (-Y^3 X^3) (a u_0) + (-1) (a u_1) +(-10 X^3) F.
\]
Thus we set $\xi_{0}^{11} := -X^3 Y^3$ and $\xi_{1}^{11} := -1$.

Next we compute a generator set of $\mathrm{syz}_I(\overline{u_0}, \overline{u_1}) \subset (k[X,Y]/I)^{2}$ by Lemma \ref{lem:ModSyz}.
For this, we compute a Gr\"{o}bner basis for $\mathrm{syz}(u_0,u_1,F)$ in $\bigoplus_{i=0}^2 S e_i $ with respect to the POT order with $e_0 \succ e_1 \succ e_2$.
The computed basis is $\{ ( Y^4 + X , X, 10),\ ( 10 X^4, Y, 0),\ \{ 0, F, 10 X^4 ) \} $.
Thus, by putting $( \eta_0^{(1)}, \eta_1^{(1)} ) := ( Y^4 + X, X )$ and $( \eta_0^{(2)}, \eta_1^{(2)} ) := ( 10 X^4, Y )$, it follows from Lemma \ref{lem:ModSyz} that $\{ ( \eta_0^{(1)}, \eta_1^{(1)} ), ( \eta_0^{(2)}, \eta_1^{(2)} ) \}$ generates $\mathrm{syz}_I(\overline{u_0}, \overline{u_1}) $.

Put $S_1 := S[T] = k[X,Y,T]$ and
\begin{eqnarray}
I_1 &:=& \langle T^2 + T + Y^3 X^3 ,\  Y T + 10 X^4,\ X T + (Y^4 + X),\ F \rangle , \nonumber 
\end{eqnarray}
respectively. and go to the second loop.
\end{enumerate}

\noindent \underline{\bf 2nd loop}: 
\begin{enumerate}
\item We have that the reduced Gr\"{o}bner basis for $J_1:=\sqrt{ I_1 + \mathrm{Jac}(I_1) }$ with respect to the lexicographical order with $T \succ Y \succ X$ is $\{ 1 \}$, and thus $J_1 = S_1$.
Hence, the algorithm terminates at this loop, and outputs $(S_1,I_1)$.
\end{enumerate}

As aresult, $A_1:=S_1/I_1$ is the integral closure of $A=S/I$ in $Q(A)$, and thus $\tilde{C}:=V( I_1 ) \subset \mathbb{A}^3 = \mathrm{Spec}(\overline{k}[X,Y,T])$ is a nonsingular variety bi-rational to $C$.
\end{example}

\section{Main Algorithm}\label{sec:main}

Let $k$ be a perfect field, and $C= V(F) \subset \mathbb{A}^2(\overline{k})$ an irreducible plane curve defined by an irreducible bivariate polynomial $F(X,Y) \in k[X,Y]$ of degree $\mathrm{deg}(F) = N$.
Let $k[C]:= A = k[X,Y] / \langle F \rangle$ be the coordinate ring of $C$, and $k(C)$ the rational function field of $C$ (i.e., the field of fractions for $k[C]$).
Note that $k[C] = k[x,y]$ and $k(C) = k(x,y)$ with $x := X \bmod{\langle F \rangle}$ and $y:= Y \bmod{\langle F \rangle}$ ($k(C)$ satisfies $\mathrm{trans}.\ \mathrm{deg}_k k(C) = 1$).
In the following, we describe a method to compute a basis of $H^0 (\tilde{C}, \varOmega_{\tilde{C}}^1)$ as a $k$-vector space, where $\tilde{C}$ is the desingularization of the Zariski closure $C'$ in $\mathbb{P}^2$ of $C$.
By Lemma \ref{lem:valid} below, we may assume the following:
\begin{itemize}
\item[{\bf (A1)}] $F(X,Y)$ is a monic polynomial in $Y$ with $\mathrm{deg}_Y F = N$.
Hence $k[x,y]$ is integral over $k[x]$.
\item[{\bf (A2)}] $y$ is a simple root of $f(Y) := F(x,Y) \in k[x][Y]$.
Hence $K'=k(x,y)\cong k(x)(y)$ is separable (and algebraic) over $k(x)$.
\end{itemize}
Putting $R = k[x]$ and $K=k(x)$, we have that $R$ is integrally closed in $K$ (since $R$ is a UFD).
It follows from the above two assumptions that $K'=k(x,y)$ is a finite separable extension of $K$.
Thus, denoting by ${\overline R}=\overline{k[x]}$ the integral closure of $R$ in $K'$, we have ${\overline R} = \overline{k[x,y]}$ since $k[x,y]$ is integral over $R=k[x]$.
As we described in Remark \ref{rem:Gorenstein}, it suffices for computing $H^0 (\tilde{C}, \varOmega_{\tilde{C}})$ to present an algorithm for computing 
\[
\mathfrak{C}_A^{(N-3)} := \left\{ \phi \in \mathfrak{C}_A : \mathrm{deg}(\phi) \leq N-3 \right\},
\]
where $\mathfrak{C}_A$ denotes the conductor of $A$ in $\overline{A}$.
We also have
\begin{equation}\label{eq:conductor}
\mathfrak{C}_A = \mathfrak{f}(\overline{R}/R[y])  = F_y \mathcal{C}_{\overline{k[x,y]}/k[x]} 
\end{equation}
by Theorem \ref{cor:Mnuk}.
Moreover, ${\overline R} = \overline{k[x,y]}$ is a free $k[x]$-module of rank $N$, by Lemma \ref{lem:PID}.
Thus, once a basis $\{ w_1, \ldots , w_N \}$ of $\overline{k[x,y]}$ as a $k[x]$-module is computed, then a basis of $\mathcal{C}_{\overline{k[x,y]}/k[x]}$ as a $k[x]$-module is also computed as its dual basis $\{ w_1^{\ast}, \ldots , w_N^{\ast} \}$, by Lemmas \ref{lem:dual_basis} and \ref{lem:dual_basis2}.
Then $\{ F_y \cdot w_i^{\ast} : 1 \leq i \leq N \}$ is a basis of $\mathfrak{C}_A $ as a $k[x]$-module.
This is also a generator set of $\mathfrak{C}_A$ as an ideal in $k[x,y]$, and hence a basis of the $k$-vector space $\mathfrak{C}_A^{(N-3)}$
is computed by the Gr\"{o}bner basis computation together with the row reduction to a Macaulay matrix, see Step B in Algorithm \ref{alg:main} below for details.

\subsection{Main algorithm}\label{subsec:main}
Here, we explicitly write down a concrete algorithm:

\begin{algorithm}[Main algorithm]\label{alg:main}
~
\begin{enumerate}
\item[{\it Input}:] An irreducible polynomial $F(X,Y) \in S=k[X,Y]$ of $\mathrm{deg}(F) = \mathrm{deg}_Y F = N$, which is monic and separable in $Y$.
\item[{\it Output}:] A basis of the truncated conductor $\mathfrak{C}_A^{(N-3)}$ as a $k$-vector space.
\end{enumerate}

Let $A = S/I$ with $I = \langle F \rangle_S$.

\noindent \underline{\bf Step A}:
Compute an explicit basis of the conductor $\mathfrak{C}_{A}$
as a $k[x]$-module, by the following three sub-steps (A-1), (A-2) and (A-3). 

\noindent \underline{{\bf Step (A-1)}: Computation of the normalization $\overline{R}$}: 

Putting $A_0 := A = S /I$, execute Algorithm \ref{alg:normalization} to compute an increasing chain of integral extensions from $A_0$ to the integral closure $\overline{A}$, say $A_0 \subset A_1 \subset \cdots \subset A_n = \overline{A}$ with $S_i = S_{i-1}[T_{i,1}, \ldots , T_{i,s_i}]$, where $T_{i,j}$'s are variables newly added in (3) of Algorithm \ref{alg:normalization} at the $i$-th loop.
Through the execution of Algorithm \ref{alg:normalization}, we also obtain the following elements:
\begin{itemize}
\item $P_i$ : A generator set of the ideal $I_{i}$ in $S_i$ corresponding to $\tilde{I}_i$ generated by \eqref{eq:GenSet} at the $i$-th loop. 
\item $Q_i$ : The set of quadratic polynomials in $P_i \subset S_i = S_{i-1}[T_{i,1}, \ldots , T_{i,s_i}]$.
\item $\phi_i : Q(A_i) \rightarrow Q( A_{i-1})$ : A homomorphism between the fields of fractions $Q(A_i)$ and $Q(A_{i-1})$ for $A_i$ and $A_{i-1}$, which is induced from the ring isomorphism $A_i \rightarrow (1/a_i) (a_i \tilde{J}_i : \tilde{J}_i) \subset Q(A_{i-1}) \ ; \ T_{i,j} \bmod{I_i} \rightarrow u_{i,j}/a_i$ given as in Lemma \ref{lem:normal} (3).
Here $(\tilde{J}_i, a_i)$ is the test pair of the $i$-th loop of Algorithm \ref{alg:normalization}, and $\{ u_{i,0} = a_i, u_{i,1}, \ldots , u_{i,s_i} \}$ is a generator set of $(a_i \tilde{J}_i : \tilde{J}_i) $ as in Lemma \ref{lem:normal}.
\end{itemize}
We set $\phi = \phi_1 \circ \cdots \circ \phi_n$.
Put $\mathcal{B}_i := \{ 1, t_{i,1}, \ldots , t_{i,s_i} \} \subset S_i / I_i$ with $t_{i,j} := T_{i,j} \bmod{I_i}$ for each $1 \leq i \leq n$, and
\begin{equation}\label{eq:gen}
\mathcal{B} := \{ t_1 \cdots t_{n} y^j : t_i \in \mathcal{B}_i, \ 0 \leq j \leq N-1 \} \subset S_n / I_n.
\end{equation}
We also order elements of $\mathcal{B}$, say write $\mathcal{B} = \{ B_1, \ldots , B_r \}$ with $r = \# \mathcal{B}$.

\noindent \underline{{\bf Step (A-2)}: Finding a basis of $\overline{R}$ as a $k[x]$-module}: 

For this, we proceed with the following three procedures:
\begin{enumerate}
\item[{\bf (A-2-1)}] Compute a generator set $\{ U_1, \ldots , U_m \} \subset  \bigoplus_{i=1}^r k[X] $ of
\begin{equation}\label{eq:eliminated_modulo_syzygy}
\mathrm{syz}_{I} (\mathcal{B}) \cap \bigoplus_{i=1}^r k[X]  = \left\{ (a_1, \ldots , a_r) \in \bigoplus_{i=1}^r k[X]  : \sum_{i=1}^r a_i B_i = 0  \mbox{ in $S_n/I_n$} \right\}
\end{equation}
as a $k[X]$-module, via the computation of Gr\"{o}bner bases and syzygy modules (details will be explained in Subsection \ref{subsec:details} below). 
\item[{\bf (A-2-2)}] Letting $U$ be the $(m \times r)$-matrix over $k[X]$ whose $i$-th row vector is $U_i$ for $1 \leq i \leq m$, compute the Smith normal form $\tilde{U}$ of $U$ together with matrices $P \in \mathrm{GL}_m (k[X])$ and $Q \in \mathrm{GL}_r(k[X])$ such that $\tilde{U} = P U Q$.
Since $S_n/I_n$ has rank $N$, the Smith normal form $\tilde{U}$ is of the form
\begin{equation}\label{eq:formU}
\tilde{U}=
\begin{pmatrix}
E_{r-N} & 0\\
0 & O_{m-(r-N), N}
\end{pmatrix},
\end{equation}
where $E_{r-N}$ (resp.\ $O_{m-(r-N),N}$) denotes the identity matrix of size $N$ (resp.\ the $(m-(r-N)) \times N$ zero matrix).
\item[{\bf (A-2-3)}] Compute $Q^{-1} \in \mathrm{GL}_r(k[X])$.
Let $W_i$ be the $(r-N+i)$-th row vector of $Q^{-1}$ for each $1 \leq i \leq N$, and put $\mathcal{W} := \{ w_i := W_i \cdot {}^t (B_1, \ldots , B_r) : 1 \leq i \leq N \}$.
\end{enumerate}

\noindent \underline{{\bf Step (A-3)}: Computation of a basis of the conductor $\mathfrak{C}_A$}: 

This step consists of the following three sub-steps:
\begin{enumerate}
\item[{\bf (A-3-1)}] For each $1 \leq i , j \leq N$, compute $\mathrm{Tr}_{k(x,y)/k(x)}(w_i w_j) \in k[x]$.
See Subsection \ref{subsec:details_3} for details.
\item[{\bf (A-3-2)}] Compute $( \mathrm{Tr}_{k(x,y)/k(x)}(w_i w_j ) )^{-1} \in \mathrm{GL}_N(k(x))$, and put
{\small
\[
\begin{pmatrix}
w_1^{\ast}\\
\vdots \\
w_N^{\ast} 
\end{pmatrix}
:=
\left( \mathrm{Tr}_{k(x,y)/k(x)}(w_i w_j ) \right)^{-1} \cdot 
\begin{pmatrix}
w_1 \\
\vdots \\
w_N
\end{pmatrix}.
\]
}
Note that $w_k^{\ast} \in k(x)[y,\{ t_{i,j} \}]$ for each $1 \leq k \leq N$.
\item[{\bf (A-3-3)}] For each $1 \leq i \leq N$, compute $\phi (w_i^{\ast} ) \in k(x,y)$, namely evaluate $u_{i,j}/u_{i,0}$ to $T_{i,j}$ successively, and compute $\mathcal{F} := \{ F_y \cdot \phi (w_i^{\ast}) : 1 \leq i \leq N \} \subset k[x,y]$.
\end{enumerate}

\noindent \underline{\bf Step B}: 
In this step, we compute a basis of $\mathfrak{C}_A^{(N-3)} := \{ \phi \in \mathfrak{C}_A : \mathrm{deg}(\phi) \leq N-3 \}$ as a $k$-vector space, by the following three procedures:
\begin{enumerate}
\item[{\bf (B-1)}] Compute a Gr\"{o}bner basis $\mathcal{G}$ of the ideal in $S=k[X,Y]$ corresponding to $\mathfrak{C}_A = \langle \mathcal{F} \rangle_A \subset A = S/I$, with respect to a graded monomial order on $S$.
\item[{\bf (B-2)}] Generate the set $\mathcal{S} := \{ m g : m \in \mathrm{Mon}(S), \ g \in \mathcal{G},\ \mathrm{deg}(mg) \leq N-3\}$, and construct its Macaulay matrix $M$ with respect to the set of all monomials in $S$ of degree $\leq N-3$.
\item[{\bf (B-3)}] Compute the (reduced) row echelon form $\mathrm{ref}(M)$ of $M$, and output the set of polynomials in $k[X,Y]$ corresponding to $\mathrm{ref}(M)$.
\end{enumerate}
\end{algorithm}


\subsection{Validity of the input}\label{subsec:validity}

Let $C_0=V(F_0)$ be a plane affine curve defined by $F_0 (X,Y) = 0$ for $F_0 \in k[X,Y]$ with $\mathrm{deg}(F_0)=D$, and assume that $F_0$ is irreducible over the algebraic closure of $k$. 
Note that we may assume $\partial F_0 / \partial Y \neq 0$, and $F_0 \notin k[Y]$ (by taking a $k$-linear coordinate change of $X$ and $Y$ if necessary).
In this subsection, we show the validity of the input of Algorithm \ref{alg:main}, more precisely, there exists a constructive isomorphism $k[C_0] \cong k[C_1]$ for some $C_1 = V(G)$ with $G(X',Y') \in k[X',Y']$ such that $G$ satisfies the same assumptions as {\bf (A1)} and {\bf (A2)} for $F$. 

\begin{lemma}\label{lem:valid}
For each integer $a \geq 1$, consider the following $k$-homomorphism
\[
\varphi_a : k [X',Y'] \rightarrow k[x,y] \ ; \ H(X',Y') \mapsto H(x-y^a,y),
\]
where $x := X \bmod{F_0}$ and $y := Y \bmod{F_0}$.
Put $G_a(X',Y') := F_0(X'+(Y')^a, Y')$.
Then we have the following:
\begin{enumerate}
\item $\varphi_a$ is surjective, in particular $\varphi_a(G_a) = F_0(x,y)$.
\item $\mathrm{Ker} (\varphi_a) = \langle G_a \rangle$, and thus $k[x',y'] = k[X',Y']/\langle G_a \rangle$ is isomorphic to $k[x,y]$, where $x' := X' \bmod{G_a}$ and $y' := Y' \bmod{G_a}$.
\item For every $a > \mathrm{deg}(F_0)$ with $a \equiv 0 \bmod{p}$,
\begin{enumerate}
\item[{\bf (A1)'}] $G_a$ is a monic polynomial in $Y'$ with $\mathrm{deg} (G_a) =\mathrm{deg}_{Y'} (G_a)$, and
\item[{\bf (A2)'}] $y'$ is a simple root of $G_a(x',Y') \in k[x'][Y']$.
\end{enumerate}
Hence $\{ x' \}$ is a separating transcendental basis of $k(x',y')$ over $k$.
\end{enumerate}
\end{lemma}

\begin{proof}
Since (1) and (2) are straightforward from the construction of $\varphi_a$, we here prove (3).
For each $(i,j)$, the highest total degree-term of $(X' + (Y')^a)^i (Y')^j$ is $(Y')^{ai + j} $, and thus
\[
G_a = \sum_{i,j} c_{i,j} (Y')^{a i + j} + H(X',Y'),
\]
where the degree in $X'$ of each term of $H(X',Y')$ is $\geq 1$, and where the total degree of $H$ is less than the maximal value of $a i + j$.
Choosing $a$ to be such that $a > \mathrm{deg}(F_0)$, we have $a i + j \neq a i' + j'$ if $(i,j) \neq (i',j')$.
Indeed, if $a i+j = a i'+j' $, then $j = j'$ by $a > \mathrm{deg}(F_0) \geq j, j'$ and thus $i=i'$.
Hence the leading coefficient of $G_a$ in $Y'$ is $c_{i,j} \neq 0$ for $(i,j)$ maximizing $ai+j$, and thus {\bf (A1)'} holds for an arbitrary $a > \mathrm{deg}(F_0)$.

It follows from $\partial F_0 / \partial Y \neq 0$ that $(\partial F_0 / \partial Y) (x, y) \neq 0$, i.e., $\partial F_0 / \partial Y \in \langle F_0 \rangle$ does not occur, by considering the total degrees of $\partial F_0 / \partial Y$ and $F_0$.
We here show $(\partial G_a / \partial Y') (x',y') \neq 0$, i.e, $\varphi_a (\partial G_a / \partial Y') \neq 0$.
Since
\begin{equation}\label{eq:partial}
\frac{\partial G_{a}}{\partial Y'} (x',y')= a (y')^{a - 1} \frac{\partial F_0}{\partial X} ( x' + (y')^{a}, y') + \frac{\partial F_0}{\partial Y} (x' + (y')^a, y'),
\end{equation}
we have $(\partial G_a / \partial Y') (x',y')  =( {\partial F_0}/{\partial Y}) (x' + (y')^a, y') $, by taking $a$ to be a multiple of $p$, and hence $\varphi_a (\partial G_a / \partial Y') = (\partial F_0 / \partial Y) (x,y) \neq 0$.
Thus {\bf (A2)'} holds for such an $a$.
\end{proof}

\begin{remark}
If $k$ satisfies $\# k > D+1$ with $D := \mathrm{deg}(F_0)$, then $\varphi_{a}$ in Lemma \ref{lem:valid} can be replaced by a linear transformation as follows:
For each $\lambda \in k$, consider the following $k$-homomorphism
\[
\varphi_{\lambda} : k [X',Y'] \rightarrow k[x,y] \ ; \ H(X',Y') \mapsto H(x-\lambda y,y).
\]
Put $G_\lambda (X',Y') := F_0(X'+\lambda Y', Y')$.
Clearly $\varphi_{\lambda}$ is surjective with $\varphi_{\lambda}(G_{\lambda}) = F_0 (x,y)$ and $\mathrm{Ker}(\varphi_{\lambda}) = \langle G_{\lambda} \rangle$, and thus $k[x',y'] = k[X', Y']/ \langle G_{\lambda} \rangle$ is isomorphic to $k[x,y]$ via a homomorphism canonically induced from $\varphi_{\lambda}$.
Similarly to the case of $G_a$, we have
\[
G_{\lambda} = \left( \sum_{i+j=D} c_{i,j} \lambda^i \right) (Y')^D + H(X',Y') = \left( \sum_{i=0}^{D} c_{i,D-i} \lambda^i \right) (Y')^{D} + H(X',Y'),
\]
where the degree in $Y'$ of each term of $H(X',Y')$ is $\leq D-1$, and where the total degree of $H$ is less than or equal to $D$.
Since $\# k > D$, there exists $\lambda \in k^{\times}$ such that $\sum_{i=0}^{D} c_{i,D-i} \lambda^i \neq 0$, and thus $G_{\lambda}$ is monic in $Y'$ for such a $\lambda$.

We also show that for a suitable $\lambda$, we have $(\partial G_{\lambda}/ \partial Y' ) (x',y') \neq 0$, that is, $\varphi_{\lambda} (\partial G_{\lambda}/ \partial Y') \neq 0$.
It follows from 
\[
(\partial G_{\lambda} / \partial Y') (X',Y')  = \lambda ({\partial F_0}/{\partial X} ) (X' + \lambda Y', Y') + ({\partial F_0}/{\partial Y}) (X' + \lambda Y', Y') 
\]
that $\varphi_{\lambda} (\partial G_{\lambda} / \partial Y') = \lambda ({\partial F_0}/{\partial X} ) (x,y) + ({\partial F_0}/{\partial Y}) (x,y) $.
Thus, the number of $\lambda$ with $\varphi_{\lambda} (\partial G_{\lambda} / \partial Y')  = 0$ is at most one.
This together with $\# k > D+1$ implies that there exists $\lambda \in k^{\times}$ such that $\sum_{i=0}^D c_{i,D-i} \lambda^i \neq 0$ and $(\partial G_{\lambda}/ \partial Y' ) (x',y') \neq 0$, as desired.
\end{remark}

\subsection{Correctness of the main algorithm}\label{subsec:correctness}
In this subsection, we prove the correctness of Algorithm \ref{alg:main} by examining that of each step.

\noindent \underline{{\bf Step (A-1)}: Computation of the normalization $\overline{R}$}: 

This step just executes Algorithm \ref{alg:normalization} for the input $(S,I)$ with $S =k[X,Y]$ and $I = \langle F \rangle$.
It follows from the following lemma that the set $\mathcal{B}$ given in \eqref{eq:gen} generates $\overline{A} = S_n / I_n$ with $A = S/I$ over $k[x]$:

\begin{lemma}
With notation same as in Step (A-1) of Algorithm \ref{alg:main}, each $S_i/I_i$ is generated by $\mathcal{B}_i := \{ 1, t_{i,1}, \ldots , t_{i,s_i} \}$ as an $S_{i-1}/I_{i-1}$-module, and it is a free $k[x]$-module of rank $N$.
\end{lemma}

\begin{proof}
The former claim follows immediately from the construction of $S_i /I_i$:
Recall from \eqref{eq:GenSet} that each $I_i$ contains elements of the form $T_{i, j} T_{i, k} - \ell_{i,j,k}$ for all $1 \leq j \leq k \leq s_i$, where $\ell_{i,j,k} \in S_{i-1}[T_{i,1}, \ldots , T_{i,s_i}]$ is linear in $T_{i,j}$ with $1 \leq j \leq s_i$.
This means that any element in $S_i$ modulo $I_i$ has no term whose degree in $T_{i,j}$ with $1 \leq j \leq s_i$ is greater than $1$, as desired.

As for the latter claim, we have $S/I \subset S_i /I_i \subset S_n/I_n = \overline{k[x,y]}$, and recall from the beginning of this section that  $\overline{k[x,y]}$ is a free $k[x]$-module of rank $N$.
Since both $S/I = k[x,y]$ and $\overline{k[x,y]}$ are of rank $N$ over the PID $k[x]$, the intermediate module $S_i /I_i$ also has rank $N$.
\end{proof}

\noindent \underline{{\bf Step (A-2)}: Finding a basis of $\overline{R}$ as a $k[x]$-module}: 

This step first computes a generator set $\mathcal{U} = \{ U_1, \ldots , U_m \} \subset  \bigoplus_{i=1}^r k[X] $ of \eqref{eq:eliminated_modulo_syzygy} as a $k[X]$-module, by Algorithm \ref{alg:elimination} in Subsection \ref{subsec:details} below. 
After computing $\mathcal{U}$, we then construct the $(m \times r)$-matrix $U$ over $k[X]$ whose $i$-th row vector is $U_i$ for $1 \leq i \leq m$, and compute its Smith normal form $\tilde{U}$ together with matrices $P \in \mathrm{GL}_m (k[X])$ and $Q \in \mathrm{GL}_r(k[X])$ such that $\tilde{U} = P U Q$.
We here show the following:

\begin{lemma}\label{lem:step2}
With notation same as in Step (A-2) of Algorithm \ref{alg:main}, the set $\mathcal{W} $ computed in (A-2-3) is a basis of $S_n / I_n$ as a free $k[x]$-module.
\end{lemma}
\begin{proof}
Let $e_i$ denote a row vector with $1$ in the $i$-th coordinate and $0$'s elsewhere.
Since the set \eqref{eq:eliminated_modulo_syzygy} is nothing but the kernel of the $k[X]$-homomorphism $\varphi_0 : \bigoplus_{i=1}^r k[X] \rightarrow S_n/I_n \ ; \ e_i \mapsto B_i$, we have the following commutative diagram:
\[
\xymatrix{
\bigoplus_{i=1}^m k[X]  \ar[r]^{f_U}  & \bigoplus_{i=1}^r k[X]  \ar[r]^{\varphi_0} \ar[d]^{f_Q} & S_n/I_n \\
\bigoplus_{i=1}^m k[X]  \ar[r]^{f_{\tilde{U}}} \ar[u]^{f_P} & \bigoplus_{i=1}^r k[X] , &
}
\]
by considering $k[X]$-homomorphisms defined by right-multiplication by matrices, e.g., $f_U$ denotes the $k[X]$-homomorphism defined by $v \mapsto v U$ for $v \in \bigoplus_{i=1}^m k[X]$.

Here, $\varphi_0$ canonically induces a $k[X]$-isomorphism from $\mathrm{Coker}(f_U)$ to $S_n / I_n$, and $\mathrm{Coker}(f_{\tilde{U}})$ is $k[X]$-isomorphic to $\mathrm{Coker}(f_U)$ via $f_Q^{-1} = f_{Q^{-1}}$.
It follows from the form of $\tilde{U}$ given in \eqref{eq:formU} that $\{ e_{r-N+i} : 1 \leq i \leq N \}$ in $\bigoplus_{i=1}^r k[X]$ gives rise to a basis of $\mathrm{Coker}(f_{\tilde{U}})$.
Since each $W_i$ is the image of $e_{r-N+i}$ by $f_{Q^{-1}}$, we have that $\{ W_1, \ldots , W_N \}$ is a basis of $\mathrm{Coker}(f_U)$, and thus $\{ w_i = \varphi_0(W_i) : 1 \leq i \leq N \}$ is also a basis of $S_n /I_n$.
\end{proof}

\noindent \underline{{\bf Step (A-3)}: Computation of a basis of the conductor $\mathfrak{C}_A$}: 

Since $\{ w_1, \ldots , w_N \}$ computed in Step (A-2) is a basis of $\overline{k[C]} = S_n / I_n$ by Lemma \ref{lem:step2}, it follows from Lemmas \ref{lem:dual_basis} and \ref{lem:dual_basis2} that $\{ w_1^{\ast}, \ldots , w_N^{\ast} \}$ is a basis of the complementary module $\mathcal{C}_{\overline{k[C]} / k[x]}$ as a $k[x]$-module.
Thus, by \eqref{eq:conductor}, the conductor $\mathfrak{C}_A$ as a $k[x]$-module is spanned by $\mathcal{F} = F_y \cdot \{ w_1^{\ast}, \ldots , w_N^{\ast} \}$.

\noindent \underline{{\bf Step B}}:

Recall from the previous paragraph that $\mathcal{F}$ computed in Step (A-3) is a basis of $\mathfrak{C}_A \subset k[x,y]$ as a $k[x]$-module, and thus it is also a generator set of $\mathfrak{C}_A$ as an ideal of $k[x,y]$.
Hence, for proving the output to be a $k$-basis of $\mathfrak{C}_A^{(N-3)}$, it suffices to show the following lemma:

\begin{lemma}\label{lem:DegreeBounded}
Let $I \subset S:=k[x_1, \ldots , x_n]$ be an ideal, and $\mathcal{G}$ its Gr\"{o}bner basis with respect to a graded monomial order $\succ$ on $x_1, \ldots , x_n$, i.e., a monomial order first comparing the total degrees of two monomials.
Then, for each integer $d \geq 0$, the (finite-dimensional) $k$-vector subspace $I_{\leq d} := \{ f \in I : \mathrm{deg}(f) \leq d \}$ of $I$ is generated by 
\begin{equation}\label{eq:gen2}
\{ m g : m \in \mathrm{Mon}(S),\ g \in \mathcal{G},\ \mathrm{deg}(mg) \leq d \}
\end{equation}
as a $k$-vector space, where $\mathrm{Mon}(S)$ denotes the set of all monomials in $S$.

Hence, a basis of $I_{\leq d}$ as a $k$-vector space can be also obtained by computing the row echelon form of the Macaulay matrix corresponding to \eqref{eq:gen2}.
\end{lemma}

\begin{proof}
Let $f \in I_{\leq d}$ be an arbitrary element, and put $\mathcal{G} = \{ g_1, \ldots , g_s \}$.
By the division algorithm together with the definition of a Gr\"{o}bner basis, there exist $h_1, \ldots , h_s \in S$ such that $f = \sum_{i=1}^s h_i g_i$, where the multi-degree of each $h_i g_i$ with respect to $\succ$ does not exceed that of $f$.
In particular, $\mathrm{deg}(h_i g_i) \leq \mathrm{deg}(f)$ for all $1 \leq i \leq s$.
Writing $h_{i} = \sum_{j=1}^{t_i} c_{i,j} m_{i,j}$ for $c_{i,j} \in k$ and $m_{i,j} \in \mathrm{Mon}(S)$ with $m_{i,1} \succ \cdots \succ m_{i,t_i}$, we have $\mathrm{LM}(m_{i,j} g_i) \preceq \mathrm{LM}(m_{i,1} g_i ) = \mathrm{LM}(h_i g_i)$, and thus $\mathrm{deg}(m_{i,j} g_i ) \leq \mathrm{deg}(f)$ for all $1 \leq j \leq t_i$.
Hence, $f = \sum_{i=1}^s \sum_{j=1}^{t_i} c_{i,j} (m_{i,j} g_i )$ with $\mathrm{deg}(m_{i,j} g_i) \leq \mathrm{deg}(f)\leq d$, as desired.
\end{proof}

By Lemma \ref{lem:DegreeBounded}, one can easily verify that the set of elements in the output modulo $F$ gives rise to a $k$-basis of $\mathfrak{C}_A^{(N-3)}$, where modulo $F$ means to take the reminder of an element by the division of $F$ as a polynomial in $Y$.



\subsection{Details of Step (A-2): Elimination on syzygy modules with modulo}\label{subsec:details}

In Step (A-2), we compute a generator set of \eqref{eq:eliminated_modulo_syzygy}.
Here we consider a more general case problem, and give a method to solve it:
Let $S = k[x_1, \ldots, x_n]$, $I \subset S$ an ideal, and $S^{(\ell)} = k[x_1, \ldots , x_{\ell}]$, where $1 \leq \ell \leq n-1$.
Given finite elements $B_1, \ldots , B_r, \ldots , B_{r+1}, \ldots , B_{r+s} \in S$ with $I = \langle B_{r+1}, \ldots , B_{r+s} \rangle_S$, we would like to compute a generator set of the $S^{(\ell)}$-module
\begin{equation}\label{eq:kernel}
 \left\{ (a_1, \ldots , a_{r}) \in \bigoplus_{i=1}^r S^{(\ell)} e_i : \sum_{i=1}^r a_i B_i = 0 \bmod{I} \right\},
\end{equation}
which is $\mathrm{syz}_I (B_1, \ldots , B_r) \cap \bigoplus_{i=1}^r S^{(\ell)} e_i$, namely, the kernel of the composition of the following maps:
\begin{equation}
\begin{CD}
\bigoplus_{i=1}^r S^{(\ell)} e_i @>\iota >> \bigoplus_{i=1}^r S e_i @>\varphi >> S @>{\bmod{I}} >> S / I , \label{eq:resolution}
\end{CD}
\end{equation}
where $\iota$ is the canonical inclusion and $\varphi $ is defined by $e_i \mapsto B_i$ for $1 \leq i \leq r$.
Here we denote by $e_i$ the vector with $1$ in its $i$-th coordinate and $0$'s elsewhere.
We also denote by $\overline{\varphi} : \bigoplus_{i=1}^r S e_i \rightarrow S/I$ the composition map of $\varphi$ and $\bmod{\; I}$.
Note that $\mathrm{syz}_I (B_1, \ldots , B_r) = \mathrm{Ker}(\overline{\varphi})$, and that \eqref{eq:kernel} is equal to $\mathrm{Ker}(\overline{\varphi}) \cap \bigoplus_{i=1}^r S^{(\ell)} e_i$.

Recall from Lemma \ref{lem:ModSyz} that, for given $\{ B_1, \ldots , B_r\}$ and $\{ B_{r+1}, \ldots , B_{r+s} \}$, a Gr\"{o}bner basis $\mathcal{G}$ of $\mathrm{Ker} (\overline{\varphi})$ with respect to an arbitrary monomial order is computed.
With this $\mathcal{G}$, a generator set of $\mathrm{Ker}(\overline{\varphi}) \cap \bigoplus_{i=1}^r S^{(\ell)} e_i $ can be also computed by the following lemma:

\begin{lemma}[Elimination theorem for free modules]\label{lem:elim}
Let $S = k[x_1, \ldots , x_n]$ be the polynomial ring of $n$ variables over a field $k$.
Let $M = \bigoplus_{i=1}^r S e_i$ be the free module over $S$ of rank $r$ with standard basis $\{ e_1, \ldots , e_r \}$.
Let $N \subset M$ be a submodule of $M$ and $\mathcal{G}$ a Gr\"{o}bner basis of $N$ with respect to $\succ_{\rm TOP}$ derived from an elimination monomial order $\succ$ on $S$ with $x_1, \ldots , x_{\ell} \prec x_{\ell+1} , \ldots , x_n$.
Put $S^{(\ell)} := k[x_1, \ldots , x_{\ell}]$ and $M_{\ell} := \bigoplus_{i=1}^r S^{(\ell)} e_i $.
Then, $\mathcal{G}_{\ell} := \mathcal{G} \cap M_{\ell}$ is a Gr\"{o}bner basis of $N_{\ell} := N \cap M_{\ell}$ with respect to the restriction of $\succ_{\rm TOP}$ to $M_{\ell}$.
\end{lemma}

\begin{proof}
It suffices to show $\langle \mathrm{LT}(\mathcal{G}_{\ell}) \rangle_{S^{(\ell)}} \supset \langle \mathrm{LT} (N_{\ell}) \rangle_{S^{(\ell)}}$.
For this, take $v \in N_{\ell} \subset N$, then there exists $w \in \mathcal{G}$ such that $\mathrm{LT}(w)$ divides $\mathrm{LT}(v)$ since $\mathcal{G}$ is a Gr\"{o}bner basis of $N$.
Namely, we can write $\mathrm{LT}(v) = m_1 e_i$ and $\mathrm{LT}(w) = m_2 e_i$ for some monomials $m_1, m_2 \in \mathrm{Mon}(S)$ with $m_2 \mid m_1$ and some $1 \leq i \leq r$.
By $v \in N_{\ell} \subset M_{\ell}$ we have $m_1 \in S^{(\ell)}$, so that $m_2 \in S^{(\ell)}$.
It follows from the definition of the TOP order that any term $m_3 e_j$ in $w$ satisfies $m_3 \in S^{(\ell)}$, and hence $w \in M_{\ell}$.
Thus $w \in \mathcal{G}_{\ell}$, as desired.
\end{proof}


Based on Lemma \ref{lem:elim}, we here write down explicit procedures to compute a generator set of \eqref{eq:kernel}:
\begin{algorithm}[Computation of a modulo-syzygy with elimination of variables]\label{alg:elimination}
~
\begin{enumerate}
\item[{\it Input}:] Finite elements $B_1, \ldots , B_r, \ldots , B_{r+1}, \ldots , B_{r+s} \in S$ and $1 \leq \ell \leq n-1$.
\item[{\it Output}:] A generator set of $\mathrm{syz}_I (B_1, \ldots , B_r ) \cap \bigoplus_{i=1}^r S^{(\ell)}$, where $I$ is an ideal in $S$ generated by $B_{r+1}, \ldots , B_{r+s}$.
\end{enumerate}

\begin{enumerate}
\item Compute a Gr\"{o}bner basis $\tilde{\mathcal{G}}$ of the syzygy module
\[
\mathrm{syz}(B_1, \ldots , B_r, B_{r+1}, \ldots , B_{r+s} ) = \left\{ (a_1, \ldots , a_{r+s}) \in \bigoplus_{i=1}^{r+s} Se_i  : \sum_{i=1}^{r+s} a_i B_i = 0 \right\},
\]
that is, the kernel of the $S$-homomorphism $\psi : \bigoplus_{i=1}^{r+s} S \longrightarrow S e_i \ ; \ e_i \mapsto B_i$.
For this, we apply Lemma \ref{lem:ComputeSyzygy}, where we use a product order $\succ_{1,2} = (\succ_1, \succ_2)$ on $S^{r+s} = \bigoplus_{i=1}^{r+s} S e_i $ with $e_i \succ_{1,2} e_j$ ($1 \leq i \leq r$, $r+1 \leq j \leq r+s$):
$\succ_1$ is a TOP order on $S^{r} = \bigoplus_{i=1}^{r} S e_i$ derived from an elimination monomial order $\succ$ on $S$ with $x_1, \ldots , x_{\ell} \prec x_{\ell+1} , \ldots , x_n$, and $\succ_2$ is an arbitrary monomial order on $S^{s} = \bigoplus_{i=r+1}^{r+s} S e_i$.
Letting $\mathrm{pr} : \bigoplus_{i=1}^{r+s} S e_i \longrightarrow \bigoplus_{i=1}^{r} S e_i$ be the projection which maps $e_i$ to $e_i$ (resp.\ $0$) for $1 \leq i \leq r$ (resp.\ for $r+1 \leq i \leq r+s$), then it follows from Lemma \ref{lem:ModSyz} that $\mathcal{G} := \mathrm{pr}(\tilde{\mathcal{G}})$ is a Gr\"{o}bner basis of $\mathrm{syz}_I (B_1, \ldots , B_r)$ with respect to $\succ_1$.
\item Put $\mathcal{G}_{\ell} := \mathcal{G} \cap M_{\ell}$, and output it.
Note that by Lemma \ref{lem:elim}, $\mathcal{G}_{\ell}$ is a Gr\"{o}bner basis of \eqref{eq:kernel} with respect to the restriction of $\succ_1$ to $\bigoplus_{i=1}^r S^{(\ell)} e_i$.
\end{enumerate}
\end{algorithm}

\subsection{Details of Step (A-3): Computation of the trace}\label{subsec:details_3}

In this step, we represent $\mathrm{Tr}_{k(x,y)/k(x)} (w_i w_j)$ as an element of $k[x]$ for each $1 \leq i, j \leq N$.
We here describe a method to find such a representation for an arbitrary element $\overline{a} \in S_n/I_n = \overline{k[C]} \subset k(x,y)$, where $\overline{a}$ denotes the residue class of an element $a$ of $S_n$ modulo $I_n$.
We denote by $\sigma_a$ a linear map defined by multiplication by $a$.
Recall that the trace of $a$ is computed as the trace of the representation matrix of $\sigma_a$ with respect to the basis $\{ w_1, \ldots , w_N \}$ for $k(x,y)$.
Thus, it suffices to find $a_{i,j} \in k[x]$ with $1 \leq  i,j \leq N$ such that $a w_i \equiv a_{i,1} w_1 + \cdots + a_{i,N} w_N \bmod{I_n}$.

We first recall that $\mathcal{B} = \{ B_1, \ldots , B_r \}$ defined in \eqref{eq:gen} generates $S_n / I_n$ over $k[x]$, and thus there exist $b_{i,j} \in k[x]$ with $1 \leq i, j \leq r$ such that $a w_i \equiv \sum_{j=1}^r b_{i,j} B_j \bmod{I_n}$.
We can find such $b_{i,j}$ as follows:
Let $\mathrm{syz}_{I_n}(a w_i, B_1, \ldots , B_r) \subset \bigoplus_{i=0}^r S_n e_i$ be a modulo-syzygy defined as in \eqref{eq:modulo-syzygy} whose Gr\"{o}bner basis $\mathcal{G}'$ with respect to an arbitrary monomial order on $S_n^{r+1}$ can be computed by Lemma \ref{lem:ModSyz}.
Compute a Gr\"{o}bner basis $\mathcal{G}$ of $\mathrm{syz}_{I_n}(a w_i, B_1, \ldots , B_r) \subset \bigoplus_{i=0}^r S_n e_i$ with respect to a TOP extension of a monomial order on $S$ with the elimination property with respect to $x$.
It follows from Lemma \ref{lem:elim} that $\mathcal{G} \cap \bigoplus_{i=0}^r k[X]$ is a Gr\"{o}bner basis of $\mathrm{syz}_{I_n}(a w_i, B_1, \ldots , B_r) \cap \bigoplus_{i=0}^r k[X]$ with respect to the TOP order induced from that for $\bigoplus_{i=0}^r S_n e_i$.
Writing $\mathcal{G} \cap \bigoplus_{i=0}^r k[X] = \{ \mathbf{g}_1, \ldots , \mathbf{g}_s \}$, it follows from $(1, b_{i,1}, \ldots , b_{i,r}) \in \mathrm{syz}_{I_n}(a w_i, B_1, \ldots , B_r) \cap \bigoplus_{i=0}^r k[X]$ that the first coordinates of $\mathbf{g}_1, \ldots , \mathbf{g}_s$ are coprime, and hence there exists $h_1, \ldots , h_s \in k[X]$ such that the sum of the first coordinates of $h_1 \mathbf{g}_1, \ldots , h_s \mathbf{g}_s$ is $1$.
Since $k[X]$ is an Euclid domain, we can compute such $h_1, \ldots , h_s$ by the extended Euclidean algorithm, and $b_{i,1}, \ldots , b_{i,r}$ are obtained by putting $(1,b_{i,1}, \ldots , b_{i,r}) = h_1 \mathbf{g}_1 + \cdots + h_s \mathbf{g}_s$.

In the following, we next consider to represent each $B_j$ as a $k[x]$-linear combination of $w_1, \ldots , w_N$, by using objects computed in Step (A-2).
Recall that each $w_i$ is computed in Step (A-2) as $w_i = \varphi_0(W_i)$, where
\[
\begin{pmatrix}
W_1 \\
\vdots \\
W_N
\end{pmatrix}
= 
\begin{pmatrix}
e_{r-N+1} \\
\vdots \\
e_r
\end{pmatrix}
\cdot Q^{-1}.
\]
Since $\{ W_i' := e_i Q^{-1} : 1 \leq i \leq r-N \}$ is a basis of $\mathrm{Im}(f_U) = \mathrm{Ker}(\varphi_0)$, we obtain an extended basis
 \[
\begin{pmatrix}
W_1' \\
\vdots \\
W_{r-N}' \\
W_1 \\
\vdots \\
W_N
\end{pmatrix}
:= 
\begin{pmatrix}
e_1 \\
\vdots \\
e_r
\end{pmatrix}
\cdot Q^{-1} 
= Q^{-1}
\begin{pmatrix}
e_1 \\
\vdots \\
e_r
\end{pmatrix}
\]
of $\bigoplus_{i=1}^r k[X]$.
It follows from $w_i = \varphi_0 (W_i)$ for $1 \leq i \leq N$ and $B_i = \varphi_0 (e_i)$ for $1 \leq i \leq r$ that
\[
\begin{pmatrix}
0 \\
\vdots \\
0 \\
w_1 \\
\vdots \\
w_N
\end{pmatrix}
:= 
Q^{-1}
\begin{pmatrix}
B_1 \\
\vdots \\
B_r
\end{pmatrix},
\]
where we also used $W_i' \in \mathrm{Ker}(\varphi_0)$ for $1 \leq i \leq r-N$.
Hence, denoting by $Q'$ the matrix obtained from $Q$ by removing its first $r-N$ columns, we have
\[
\begin{pmatrix}
\sigma_a (w_1) \\
\vdots \\
\sigma_a (w_N)
\end{pmatrix}
=
\begin{pmatrix}
a w_1 \\
\vdots \\
a w_N
\end{pmatrix}
=
\begin{pmatrix}
b_{1,1} & \cdots & b_{1,r} \\
\cdots & & \cdots \\
b_{N,1} & \cdots & b_{N,r}
\end{pmatrix}
\begin{pmatrix}
B_1 \\
\vdots \\
B_r
\end{pmatrix}
=
\begin{pmatrix}
b_{1,1} & \cdots & b_{1,r} \\
\cdots & & \cdots \\
b_{N,1} & \cdots & b_{N,r}
\end{pmatrix}
Q'
\begin{pmatrix}
w_1 \\
\vdots \\
w_N
\end{pmatrix}.
\]
Therefore, the representation matrix of $\sigma_a$ with respect to $\{ w_1, \ldots , w_N \}$ is {${}^t (B Q')$}, and hence $\mathrm{Tr}_{k(x,y)/k(x)}(\overline{a})$ is computed as the trace of the matrix ${}^t (B Q')$, by the following procedures:

\begin{enumerate}
\item For each $1 \leq i \leq N$, find $b_{i,j} \in k[x]$ with $1 \leq j \leq r$ such that $a w_i \equiv \sum_{j=1}^r b_{i,j} B_j \bmod{I_n}$, by the Gr\"{o}bner basis computation for free modules.
Let $B$ denote the $N \times r$ matrix over $k[x]$ whose $(i,j)$ entry is $b_{i,j}$.
\item Denoting by $Q'$ the matrix obtained from $Q$ by removing its first $r-N$ columns, compute ${}^t (B Q')$, and then its trace is nothing but $\mathrm{Tr}_{k(x,y)/k(x)}(\overline{a})$.
\end{enumerate}

\section{Implementation and example}\label{sec:implementation}

We implemented Algorithm \ref{alg:main} over Magma~\cite{Magma} V2.25-3 in its 64bit version, and the source code will be available at 
\begin{center}{\small
\texttt{https://sites.google.com/view/m-kudo-official-website/english/code/rdf}.
}
\end{center}
In the code, the function \texttt{SpaceOfDifferentialForms} implements Algorithm \ref{alg:main}.
Note that, as for the Gr\"{o}bner basis computation of our implementation, we utilize Magma's built-in functions (e.g., \texttt{Groebner}) where Faugere's $F_4$ algorithm and its variants are efficiently implemented.
Here, we demonstrate Algorithm \ref{alg:main} in the following example (Example \ref{ex:trigonal}), where we used our implementation.

\begin{example}\label{ex:trigonal}
Consider a plane singular curve $C = V(F)$ in $\mathbb{A}^2(\overline{k})$ over $k:=\mathbb{F}_{11}$ defined by $F := X^5 + Y^5 + X Y$, whose normalization is known to be a superspecial curve of genus $5$ over $\mathbb{F}_{11}$, see \cite{trigonal} for details.
Here let us compute an explicit basis of $H^0(\tilde{C},\varOmega_{\tilde{C}})$ as a $k$-vector space.
Note that $C'$ is singular only at $(0:0:1)$, and thus it suffices from Theorem \ref{thm:Gorenstein} to compute a basis of $\mathfrak{C}_A^{(N-3)}$ by Algorithm \ref{alg:normalization}.

\noindent \underline{{\bf Step (A-1)}: Computation of the normalization}: 

We execute Algorithm \ref{alg:normalization} for the input $(S,I)$ with $S = k[X,Y]$ and $I = \langle F \rangle_S$.
Recall from Example \ref{ex:normalization} that Algorithm \ref{alg:normalization} terminates at the first step of its second loop, and an integral closure of $A=S/I$ is obtained as $A_1 = S_1 / I_1$, where $S_1 = k[X,Y,T] $ and $I_1 = \langle P_1 \rangle_{S_1}$ with
\begin{eqnarray}
P_1 &:=& \{ T^2 + T + Y^3 X^3, Y T + 10 X^4, X T + Y^4 + X,  F \}, \nonumber 
\end{eqnarray}
together with $Q_1 = \{ T^2 + T + Y^3 X^3 \}$ and $\phi_1 : A_1 \rightarrow k(x,y)$ mapping $(x, y, t)$ to $(x,y, x^4/y)$.

In this case, the generator set \eqref{eq:gen} of $S_1/I_1$ as a $k[x]$-module is
\[
\mathcal{B} = \{ t y^4, t y^3, t y^2, t y, t, y^4, y^3, y^2, y, 1\},
\]
where we set $t := T \bmod{I_1}$.
Put $(B_1, \ldots , B_{10} ) = (t y^4, t y^3, t y^2, t y, t, y^4, y^3, y^2, y, 1)$.

\noindent \underline{{\bf Step (A-2)}: Finding a basis of $\overline{R}$ as a $k[x]$-module}: 

\begin{enumerate}
\item[{\bf (A-2-1)}] Compute a generator set of \eqref{eq:eliminated_modulo_syzygy} as a $k[X]$-module, say
\[
\mathrm{syz}_{I_1}(\mathcal{B}) \cap \bigoplus_{i=1}^{10} k[X] = \left\{ (a_1, \ldots , a_{10}) \in \bigoplus_{i=1}^{10} k[X] : \sum_{i=1}^{10}a_i B_i = 0 \bmod{I_1} \right\}.
\]
For this, we execute Algorithm \ref{alg:elimination} for the inputs $B_1, \ldots , B_{10}$ as above, $B_{11} = T^2 + T + Y^3 X^3$, $B_{12} = Y T + 10 X^4$, $B_{13} = X T + Y^4 + X$, $B_{14} = F$, and $\ell:=1$.
As a result,
\[
U = 
\begin{pmatrix}
U_1 \\
U_2 \\
U_3 \\
U_4 \\
U_5 \\
\end{pmatrix}
=
\begin{pmatrix}
10 & 0 & 0 & 0 & 0 & 0 & X^4 & 0 & 0 & 0  \\
0 & 10 & 0 & 0 & 0 & 0 & 0 & X^4 & 0 & 0  \\
0 & 0 & 10 & 0 & 0 & 0 & 0 & 0 & X^4 & 0  \\
0 & 0 & 0 & 10 & 0 & 0 & 0 & 0 & 0 & X^4  \\
0 & 0 & 0 & 0 & X & 1 & 0 & 0 & 0 & X  \\
\end{pmatrix}
\]
is a matrix whose row vectors generate $\mathrm{syz}_{I_1}(\mathcal{B}) \cap \bigoplus_{i=1}^{10} k[X]$.
\item[{\bf (A-2-2)}] The computed Smith normal form $\tilde{U}$ of $U$ and transformation matrices $P \in \mathrm{GL}_{5} (k[X])$ and $Q \in \mathrm{GL}_{10}(k[X])$ are as follows:
$P = E_5$, $\tilde{U} = \begin{pmatrix} E_5 & O_5 \end{pmatrix}$, where $E_5$ (resp.\ $O_5$) is the identity (resp.\ zero) matrix of size $5$, and
\begin{equation}
Q=
\begin{pmatrix}
10 & 0 & 0 & 0 & 0 & X^4 & 0 & 0 & 0 & 0  \\
0 & 10 & 0 & 0 & 0 & 0 & X^4 & 0 & 0 & 0  \\
0 & 0 & 10 & 0 & 0 & 0 & 0 & X^4 & 0 & 0  \\
0 & 0 & 0 & 10 & 0 & 0 & 0 & 0 & X^4 & 0  \\
0 & 0 & 0 & 0 & 0 & 0 & 0 & 0 & 0 & 10  \\
0 & 0 & 0 & 0 & 1 & 0 & 0 & 0 & 10 X & X  \\
0 & 0 & 0 & 0 & 0 & 1 & 0 & 0 & 0 & 0  \\
0 & 0 & 0 & 0 & 0 & 0 & 1 & 0 & 0 & 0  \\
0 & 0 & 0 & 0 & 0 & 0 & 0 & 1 & 0 & 0  \\
0 & 0 & 0 & 0 & 0 & 0 & 0 & 0 & 1 & 0  \\
\end{pmatrix}.
\nonumber
\end{equation}
\item[{\bf (A-2-3)}] We have 
\begin{equation}
Q^{-1}=
\begin{pmatrix}
10 & 0 & 0 & 0 & 0 & 0 & X^4 & 0 & 0 & 0  \\
0 & 10 & 0 & 0 & 0 & 0 & 0 & X^4 & 0 & 0  \\
0 & 0 & 10 & 0 & 0 & 0 & 0 & 0 & X^4 & 0  \\
0 & 0 & 0 & 10 & 0 & 0 & 0 & 0 & 0 & X^4  \\
0 & 0 & 0 & 0 & X & 1 & 0 & 0 & 0 & X  \\
0 & 0 & 0 & 0 & 0 & 0 & 1 & 0 & 0 & 0  \\
0 & 0 & 0 & 0 & 0 & 0 & 0 & 1 & 0 & 0  \\
0 & 0 & 0 & 0 & 0 & 0 & 0 & 0 & 1 & 0  \\
0 & 0 & 0 & 0 & 0 & 0 & 0 & 0 & 0 & 1  \\
0 & 0 & 0 & 0 & 10 & 0 & 0 & 0 & 0 & 0  \\
\end{pmatrix}. \nonumber
\end{equation}
Thus, putting $( w_1, w_2, w_3 , w_4, w_5 ) = (y^3, y^2, y, 1, 10 t)$, one also has that $\mathcal{W} = \{ w_1, \ldots , w_5 \}$ is a basis of $S_1/I_1$ as a $k[x]$-module.
\end{enumerate}

\noindent \underline{{\bf Step (A-3)}: Computation of a basis of the conductor $\mathfrak{C}_A$}: 
\begin{enumerate}
\item[{\bf (A-3-1)}] For each $1 \leq i , j \leq 5$, compute $\mathrm{Tr}_{k(x,y)/k(x)}(w_i w_j) \in k[x]$.
We here show details on the computation only for $\mathrm{Tr}_{k(x,y)/k(x)}(w_1^2)$.
By computing a generator set of $\mathrm{syz}_{I_1} (w_1^2, B_1, \ldots , B_{10}) \cap \bigoplus_{i=0}^{10} k[X]$ with a method described in Subsection \ref{subsec:details_3}, we have a relation
\[
\begin{pmatrix}
w_1^2 w_1 \\
w_1^2 w_2 \\
w_1^2 w_3 \\
w_1^2 w_4 \\
w_1^2 w_5
\end{pmatrix}
\equiv
\begin{pmatrix}
0      & 0    &  0 & x^{2} & 0 & 10 x^{5} & 0 & 0 & x^2 & 0\\
10x      & 0    &  0 & 0 & 0 & 10 x & 0 & 0 & 0 & 0\\
0      & 10x    &  0 & 0 & 0 & 0 & 10x & 0 & 0 & 0\\
0      & 0    &  10x & 0 & 0 & 0 & 0 & 10 x & 0 & 0\\
0      & 0    &  x & x^{5} & 0 & 0& 0 & 0 & 0 & 0\\
\end{pmatrix}
\cdot 
\begin{pmatrix}
B_1\\
\vdots \\
B_{10} 
\end{pmatrix} 
\]
modulo $I_1$.
Let $Q'$ be the matrix obtained from $Q$ by removing its first $5$ columns.
Since ${}^t (B_1, \ldots , B_{10}) = (Q') {}^t (w_1, \ldots , w_{5})$, we have
\begin{equation}
\begin{pmatrix}
w_1^2 w_1 \\
w_1^2 w_2 \\
w_1^2 w_3 \\
w_1^2 w_4 \\
w_1^2 w_5
\end{pmatrix}
\equiv
\begin{pmatrix}
 0 & 0 & x^2 & 2 x^6 & 10 x^6  \\
 10 x^5 & 0 & 0 & x^2 & 10x^2  \\
 10 x& 10x^5 & 0 & 0 & 0  \\
 0 & 10x & 10x^5 & 0 & 0  \\
 0 & 0 & x^5 & x^9 & 0  \\
\end{pmatrix}
\begin{pmatrix}
w_1 \\
 w_2 \\
w_3 \\
w_4 \\
 w_5
\end{pmatrix},\nonumber
\end{equation}
and hence $\mathrm{Tr}_{k(x,y)/k(x)}(w_1^2)= 0$.

As a result, the computed matrix of traces is
\[
\left( \mathrm{Tr}_{k(x,y)/k(x)}(w_i w_j ) \right) =
\begin{pmatrix}
0       & 6 x^5     & 7 x      & 0 & 0\\
6 x^5  & 7 x        & 0        & 0 & 0\\
7 x      & 0          & 0        & 0 & 6 x^4\\
0       &  0   &  0     &  5 & 1\\
0     &   0   &  6 x^4     & 1   & 1
\end{pmatrix}.
\]
\item[{\bf (A-3-2)}] Computing $( \mathrm{Tr}_{k(x,y)/k(x)}(w_i w_j ) )^{-1} \in \mathrm{GL}_5(k(x))$, we have
{\small
\begin{eqnarray}
& & 
{}^t \!
\begin{pmatrix}
w_1^{\ast} & w_2^{\ast} & w_3^{\ast} & w_4^{\ast} & w_5^{\ast} 
\end{pmatrix}
=
\left( \mathrm{Tr}_{k(x,y)/k(x)}(w_i w_j ) \right)^{-1} \cdot 
{}^t \! 
\begin{pmatrix}
y^3 & y^2 & y & 1 & 10 t
\end{pmatrix}
\nonumber \\
&=&
\frac{1}{x^{15} + 3}   
\begin{pmatrix}
5x^6      & 2x^{10}    &  \frac{2}{x} & 3x^{3} & 7 x^3\\
2x^{10} & \frac{2}{x}    & 3x^{3}  & 10 x^7 & 5x^7 \\
\frac{2}{x} & 3x^{3} & 10 x^7 & 4x^{11} & 2x^{11} \\
3x^{3} & 10 x^7  &  4x^{11} & 9x^{15}+9 & 2\\
7x^3         & 5x^7        & 2x^{11} & 2  & 1
\end{pmatrix}
\cdot 
\begin{pmatrix}
y^3\\
 y^2 \\
 y \\
1 \\
10 t
\end{pmatrix}, \nonumber 
\end{eqnarray}
}
where $w_i^{\ast} \in k(x)[y,t]$ for each $1 \leq i \leq 5$.
\item[{\bf (A-3-3)}] Substituting $x^4/y$ into $t$ in $w_i^{\ast}$ for $1 \leq i \leq 5$, we have
\[
F_y \cdot
\begin{pmatrix}
w_1^{\ast} & \cdots & w_5^{\ast}
\end{pmatrix}
\equiv
\begin{pmatrix}
y & y^2 & y^3 & y^4 & x
\end{pmatrix}
\pmod{F}.
\]
Thus, $\{ y,y^2, y^3, y^4, x \}$ spans $\mathfrak{C}_A$ as a $k[x]$-module, and hence it also generates $\mathfrak{C}_A$ as an ideal (i.e., as a $k[x,y]$-module).
\end{enumerate}

\noindent \underline{\bf Step B}:
Finally, we compute a basis of $\mathfrak{C}_A^{(2)} := \{ \phi \in \mathfrak{C}_A : \mathrm{deg}(\phi) \leq 2 \}$ as a $k$-vector space, by the following three procedures:
\begin{enumerate}
\item[{\bf (B-1)}] A computed Gr\"{o}bner basis of the ideal $\langle Y, Y^2, Y^3, Y^4, X, F(X,Y) \rangle_{k[X,Y]}$ with respect to the graded lexicographic order with $X \succ Y$ is $\mathcal{G} = \{ X, Y \}$.
\item[{\bf (B-2)}] Putting $\mathcal{S} := \{ m g : m \in \mathrm{Mon}(k[X,Y]), \ g \in \mathcal{G},\ \mathrm{deg}(mg) \leq 2\}$, we have $\mathcal{S} = \{ Y^2, X Y, X^2, Y, X\}$.
The Macaulay matrix $M$ corresponding to $\mathcal{S}$ with respect to the set of monomials of degree $\leq 2$ is as follows:
\begin{eqnarray*}
\begin{split}
 M =
  \kbordermatrix{
        & X^2 & X Y & Y^2 & X & Y & 1 \\
     Y^2 & 0 & 0 & 1 & 0 & 0 & 0\\
     XY & 0 & 1 & 0 & 0 & 0 & 0\\
     X^2 & 1 & 0 & 0 & 0 & 0 & 0\\
     Y & 0 & 0 & 0 & 0 & 1 & 0\\
     X & 0 & 0 & 0 & 1 & 0 & 0\\
    }.
\end{split}    
\end{eqnarray*}
\item[{\bf (B-3)}] The reduced row echelon form $\mathrm{ref}(M)$ of $M$ is $\begin{pmatrix} E_5 & {\bf 0} \end{pmatrix}$, where $E_5$ (resp.\ ${\bf 0}$) is the identity matrix of size $5$ (resp.\ the zero column vector of dimension $5$).
Thus the set of its corresponding polynomials is $\{ x^2, x y, y^2, x, y \}$, which gives rise to a basis of $\mathfrak{C}_A^{(2)} $ as a $k$-vector space.
\end{enumerate}

Therefore, it follows from Theorem \ref{thm:Gorenstein} that
\[
\left\{ \frac{x^2}{y^4 + x} dx, \frac{x y}{y^4 + x} dx,  \frac{y^2}{y^4 + x} dx, \frac{x}{y^4 + x} dx, \frac{y}{y^4 + x} dx \right\}
\]
spans $H^0 (\tilde{C}, \varOmega_{\tilde{C}})$ as a $k$-vector space.



\end{example}


\section{Computing Cartier-Manin matrices}\label{sec:Cartier-Manin}
Let $C=V(F)$ be a plane curve with $\deg(F)=N$.
Let $C'$ be the the Zariski closure of $C$ in ${\mathbb P}^2$.
Let $\pi: \tilde C \to C'$ be the normalization.
The aim of this section is to describe the Cartier operator (Cartier-Manin matrix)
on $H^0({\tilde C}, \varOmega_{\tilde C})$
and to give a remark on computation of a basis of $H^1({\tilde C}, {\mathcal O}_{\tilde C})$.

In the previous sections, we gave an algorithm
to obtain a basis of $H^0({\tilde C}, \varOmega_{\tilde C})$.
Let ${\mathcal V}$ be the Cartier operator on $H^0({\tilde C}, \varOmega_{\tilde C})$
(usually the symbol $\mathcal C$ is used for the Cartier operator,
but this symbol is used for complementary modules in Section \ref{subsec:complementary}).
We can calculate the Cartier-Manin matrix of $\tilde C$
with respect to the basis.
We recall the result \cite[Theorem 1.1]{SV} by St\"ohr and Voloch:
\begin{equation}\label{formula:Stohr-Voloch}
{\mathcal V} \left(\phi\frac{dx}{F_y}\right)
=\left(\frac{\partial^{2p-2}}{\partial x^{p-1}\partial y^{p-1}}
(F^{p-1}\phi)\right)^{1/p} \frac{dx}{F_y}.
\end{equation}
\begin{example}\label{exam:Cartier-Manin}
\begin{enumerate}
\item[(1)] Consider $C=V(F)$ with $F=x^5+y^5+xy$ over ${\mathbb F}_{11}$ as in Example \ref{ex:trigonal}. As we have seen,
as a basis of $H^0({\tilde C}, \varOmega_{\tilde C})$ we have
$\left\{\left.\phi_i\frac{dx}{F_y}\ \right|\ i=1,\ldots,5\right\}$
with $(\phi_1,\ldots,\phi_5)=(x^2,xy,y^2,x,y)$.
It is straightforward to see that the Cartier matrix of $\tilde C$ is zero, i.e.,
$\tilde C$ is superspecial.
\item[(2)]
Consider $C=V(F)$ with $F=x^5+y^5 + (x+y)^3 + xy$ over $k={\mathbb F}_2$. 
This curve is singular at $(x,y)=(0,0), (1,0), (0,1)$ and $C'$ is regular at the infinite points.
By executing an implementation of our algorithm in Section \ref{sec:main},
$H^0({\tilde C}, \varOmega_{\tilde C})$ has a basis
$\left\{\left.\phi_i\frac{dx}{F_y}\ \right|\ i=1,2,3\right\}$ with $\phi_1 = x^2+x$, $\phi_2=xy$ and $\phi_3=y^2+y$.
A straightforward computation, for example $(\partial^{2}F\phi_1/\partial x\partial y)^{1/2} = (y^4 + y^2)^{1/2} = \phi_3$, shows that the Cartier-Manin matrix of $\tilde C$ with respect to this basis is
\[
\begin{pmatrix}
0 & 0 & 1\\
0 & 0 & 0\\
1 & 0 & 0
\end{pmatrix}.
\]
\item[(3)] Consider $C=\Spec(A)$, where $A=\F_2[x,y]/\langle F \rangle$ with $F=y^7-x^2(x-1)^2$.
Note that $C'$ is singular at every $\F_2$-rational point,
which means that any $\F_2$-linear transformation can not make the infinity regular.
By executing an implementation of our algorithm, we see that ${\frak C}_A^{(4)}$ has a basis
\[
\{x^4 + x,\ 
x^3y + xy,\ 
x^2y^2 + xy^2,\ 
xy^3,\ 
y^4,\ 
x^3 + x,\ 
x^2y + xy,\ 
y^3,\ 
x^2 + x\} .
\]
The $\F_2$-subspace of ${\frak C}_A^{(4)}$ where the associated differential form is regular at the infinite point is generated by $y^4$, $y^3$, $x^2+x$
(by looking at the conductor at the infinite point or at ${\frak C}_{A'}^{(4)}$, cf.\ Remark \ref{rem:Gorenstein}).
Thus, a basis of $H^0(\tilde C, \varOmega_{\tilde C})$ can be taken as $1/y^2 dx$, $1/y^3 dx$ and $(x^2+x)/y^6 dx$.
The Cartier-Manin matrix of $\tilde C$ with respect to this basis is
\[
\begin{pmatrix}
0 & 0 & 0\\
0 & 0 & 1\\
0 & 0 & 0
\end{pmatrix}
\]
by a straightforward computation, for example ${\mathcal V}\left((x^2+x)/y^6 dx\right) = 1/y^3 dx$.
\end{enumerate}
\end{example}

Finally, we mention about $H^1(\tilde C, {\mathcal O}_{\tilde C})$.
This is just the dual notion of $H^0(\tilde C, \varOmega_{\tilde C})$,
but it is important to have an explicit basis of $H^1(\tilde C, {\mathcal O}_{\tilde C})$.
Let $U_0, U_1$ and $U_2$ be the affine open subshemes of $C'$ 
corresponding to the parts of $X\ne 0$, $Y\ne 0$ and $Z\ne 0$ respectively.
The short exact sequence
\[
\begin{CD}
0 @>>> {\mathcal O}_{C'} @>>> \pi_* {\mathcal O}_{\tilde C} @>>> \pi_* {\mathcal O}_{\tilde C}/{\mathcal O}_{C'}  @>>> 0
\end{CD}
\]
induces
\[
\begin{CD}
0 @>>> H^0(\pi_* {\mathcal O}_{\tilde C}/{\mathcal O}_{C'}) @>>> H^1(C',{\mathcal O}_{C'}) @>>> H^1({\tilde C}, {\mathcal O}_{\tilde C}) @>>> 0.
\end{CD}
\]
From the short exact sequence
\[
\begin{CD}
0 @>>> {\mathcal I} @>>> {\mathcal O}_{{\mathbb P}^2} @>>> {\mathcal O}_{{\mathcal C}'} @>>> 0
\end{CD}
\]
with the ideal sheaf ${\mathcal I}$ defining $C'$, we have
\[
H^2({\mathbb P}^2, \mathcal{O}(-N))\overset{F\times}{\simeq} H^2({\mathbb P}^2, {\mathcal I}) \simeq H^1(C',{\mathcal O}_{C'}) \simeq H^0({C'},{\mathcal K}/{\mathcal O}_{C'})/K,
\]
where $\mathcal K$ is the constant sheaf of rational functions on $C'$.
The class of $1/x^{\ell}y^{m}z^{n}$ with $\ell, m, n\ge 1$ and $\ell + m + n=N$
in $H^2({\mathbb P}^2, \mathcal{O}(-N))$
is sent to the class of $F/x^\ell y^m z^n$,
which can be written as $f_{ij}-f_{ik}+f_{jk}$ for some
$f_{ij}\in \Gamma(U_i\cap U_j, {\mathcal O}_{C'})$.
Then $\{f_{ij}\}$ gives a class of $H^1(C',{\mathcal O}_{C'})$.
For a fixed $r_0\in H^0(U_0,{\mathcal K}/{\mathcal O}_{C'})$,
we put $r_i:=f_{i0}+r_0$; then this defines an element of $H^0(C',{\mathcal K}/{\mathcal O})$, which is called a {\it repartition}.
The pairing
\[
\langle\ ,\ \rangle:\quad  
H^1(C', {\mathcal O}_{C'}) \otimes   H^0({\tilde C},\varOmega_{\tilde C}) \simeq 
H^0(C', {\mathcal K}/{\mathcal O})/K \otimes   H^0({\tilde C},\varOmega_{\tilde C}) \to k
\]
is defined by taking their residues,
and it is embedded into the commutative diagram
\[\begin{CD}
H^1(C', \mathcal O_{C'}) @. \otimes @. H^0(C',\omega_{C'}^o) @>>> k\\
@VVV @. @AAA @|\\
H^1({\tilde C}, \mathcal O_{\tilde C}) @. \otimes @. H^0({\tilde C},\varOmega_{\tilde C}) @>>> k,
\end{CD}\]
where $\omega_C^o$ is the dualizing sheaf on $C$.
In particular, the image of $H^0(\pi_* {\mathcal O}_{\tilde C}/{\mathcal O}_{C'}) \to H^1(C',{\mathcal O}_{C'})$
is equal to
\[
H^0({\tilde C},\varOmega_{\tilde C})^\perp := 
\left\{ \eta\in H^1(C',{\mathcal O}_{C'}) \ \left|\ \langle \eta , \omega\rangle = 0
\text{ for all } \omega\in H^0({\tilde C},\varOmega_{\tilde C}) \right.\right\}.
\]

\begin{lemma}
With identification $F\cdot :H^2({\mathbb P}^2, \mathcal{O}(-N))\simeq H^1(C',{\mathcal O}_{C'})$, we have
\[
H^1({\tilde C},{\mathcal O}_{\tilde C})
\simeq F\cdot H^2({\mathbb P}^2, \mathcal{O}(-N))/H^0({\tilde C},\varOmega_{\tilde C})^\perp.
\]
\end{lemma}
This lemma gives us a feasible algorithm to give a basis of $H^1({\tilde C},{\mathcal O}_{\tilde C})$. In a future work, we shall give an algorithm determining the Ekedahl-Oort type of $\tilde C$
based on these descriptions of $H^0({\tilde C}, \varOmega^1_{\tilde C})$ and 
$H^1({\tilde C},{\mathcal O}_{\tilde C})$ using bases on them discussed in the present paper (cf.\ see the recent work \cite{Moonen} by Moonen, for a way to compute the Ekedahl-Oort type of nonsingular complete-intersection curves).

\section{Concluding Remarks}\label{sec:conc}
In this paper, we proposed an algorithm to compute an explicit basis of the space of regular differential forms on a general plane curve, whereas most of previous works study particular cases such as hyperelliptic curves. 
More precisely, we reduced the problem into computing a basis of the conductor of the coordinate ring of the curve, based on Gorenstein's work~\cite{Gorenstein}.
As for the computation of the conductor, we also confirmed that M{\v n}uk's framework~\cite{Mnuk} works well even in the positive characteristic case, and realized it as an algorithm via the theory of Gr\"{o}bner bases for free modules.
By our algorithm together with St\"ohr-Voloch's formula in \cite{SV} for computing the Cartier operator, we can symbolically compute the Cartier-Manin matrix of an arbitrary plane curve.
Our algorithm was also demonstrated for concrete examples, which will be helpful for the further use of our algorithm.

While our algorithm might derive fruitful applications such as computing invariants (e.g., Ekedahl-Oort type), the computational complexity has not been determined yet neither in theory nor in practice, due to the difficulty of estimating that of computing Gr\"{o}bner bases for free modules.
Let us leave the complexity analysis of our algorithm as future work, where it is important to find some mathematical invariants for measuring the complexity.

\subsection*{Acknowledgments}
This work was supported by JSPS Grant-in-Aid for Young Scientists 20K14301, and JSPS Grantin-Aid for Scientific Research (C) 21K03159.
This work was also supported by JST CREST Grant Number JPMJCR2113, Japan.

\end{document}